\documentclass[12pt]{article}
\usepackage{geometry}
\usepackage{amsthm}
\usepackage{amsmath}
\usepackage{amssymb}
\usepackage{amsfonts}
\usepackage{array}

\theoremstyle{plain}
\newtheorem{theorem}{Theorem}
\newtheorem{proposition}{Proposition}

\newtheorem{lemma}{Lemma}

\theoremstyle{definition}
\newtheorem{definition}{Definition}

\DeclareMathOperator{\loc}{loc}
\DeclareMathOperator{\dist}{dist}
\DeclareMathOperator{\diam}{diam}

\DeclareMathOperator{\kernel}{ker}

\DeclareMathOperator{\Image}{Im}

\begin{document}

\title{Sharp Geometric Rigidity of Isometries \\ on Heisenberg Groups
}

\author{D.~V.~Isangulova and S.~K.~Vodopyanov\thanks{ The research was partially
supported by the Russian Foundation for Basic Research
(Grant 10--01--00662), the State Maintenance Program for Young Russian
Scientists and the Leading Scientific Schools of the Russian Federation
(Grant NSh~921.2012.1).}}
\date{\empty}
\maketitle

\begin{abstract}We prove sharp geometric rigidity estimates for isometries on Heisenberg groups.
Our main result asserts that every
$(1+\varepsilon)$-quasi-isometry on a~John domain of the Heisenberg group
$\mathbb{H}^n$, $n>1$, is close to some isometry up to proximity order  $\sqrt{\varepsilon}+\varepsilon$ in the uniform norm, and  up to proximity order $\varepsilon$ in the $L_p^1$-norm. We give examples showing the asymptotic sharpness of our results.
\end{abstract}

\section{Introduction}

The following question is studied in elasticity theory: what can we say about a~global
deformation of a~rigid body provided that local deformations are small?
This question leads to the mathematical problem \cite{john} formulated below.

A deformation is interpreted as a~homeomorphism
$f\colon U\to \mathbb{R}^3$, where
$U$ is an~open set in~$\mathbb{R}^3$.
The Jacobi matrix $Df(x)$  is assumed to exist almost everywhere.
The symmetric matrix
$E(x)=\frac{1}{2}((Df(x))^t Df(x)-I)$
determines $Df(x)$ up to an~orthogonal matrix. The matrix $E(x)$
is associated to the deformation or strain tensor (see, for example, \cite{lan-lif}).
The notion of deformation tensor $E$ plays a~key role in elasticity
theory (see \cite{lan-lif} for instance):
various full or partial linearization problems there
are based on the assumption that the deformation tensor is sufficiently small.
How can this assumption affect $f(x)$ itself? It is known that if
$E(x)=0$ almost everywhere on $U$ then $f$ is a~rigid motion under the condition
of sufficient regularity. If $E$ is small on $U$ in some sense then
what is the global difference between $f$ and a~rigid motion on the entire
domain? If the difference is small globally then this property is called
geometric rigidity of isometries.

If $f$ is a~homeomorphism with small $E(x)$ then $f$ is locally bi-Lipschitz
(see \cite{resh} for instance). This leads to a~natural interpretation of deformations
as bi-Lipschitz mappings.

In 1961 F.~John studied this question in a~more general setting;
namely, he considered a~mapping
$f\colon U\to \mathbb{R}^n$,
where $U$ is an~open set in $\mathbb{R}^n$.
He showed that
\textit{for a~locally $(1+\varepsilon)$-bi-Lipschitz mapping $f$,
where $\varepsilon<1$, there exists a~motion~$\varphi$
satisfying}
\begin{equation}
\label{eq:Sob_stab_Eucl}
\|Df-D\varphi\|_{p,U}
\leqslant C_1 p\varepsilon |U|^{1/p}
\end{equation}
\textit{and}
\begin{equation}
\label{eq:uni_stab_Eucl}
\sup_{x\in U}|f(x)-\varphi(x)|\leqslant C_2\diam(U) \varepsilon.
\end{equation}
F.~John established \eqref{eq:uni_stab_Eucl} for a~domain $U$ of a~special kind,
now called a~John domain, and \eqref{eq:Sob_stab_Eucl} on cubes.
Later Yu.~G.~Reshetnyak \cite{resh}
established \eqref{eq:Sob_stab_Eucl}  and \eqref{eq:uni_stab_Eucl}
on John domains without constraints on $\varepsilon$ using a~different method.

John also studied the question of geometric rigidity under small
integral deviations of the deformation tensor
\cite{john1}:
\textit{if $U$ is a~cube, $f\colon U\to \mathbb{R}^n$ is a~mapping of class
$C^1$, and $\sup|E(x)|$ on $U$ is less than a~fixed number then there
exists a~motion ~$\varphi$ such that
\begin{equation}
\label{eq:Sob_stab_integr}
\|Df-D\varphi\|_{p,U}\leqslant C_3 \|E\|_{p,U}
\quad
\text{if }p>1
\end{equation}
and
$$
\sup_{x\in U}|f(x)-\varphi(x)|\leqslant C_4\diam(U) \|E\|_{p,U}
\quad
\text{if }p>n.
$$
}Recently \cite{FJM} Friesecke, James, and M\"uller have demonstrated that
\eqref{eq:Sob_stab_integr} holds for every  Sobolev mapping of class
$W_p^1$ on a~Lipschitz domain $U$ without constraints on
$$
\sup_{x\in U}|E(x)|=\sup_{x\in U}\dist(Df(x),\mathrm{SO}(n)).
$$

Note that the geometric rigidity problem has a~much wider interpretation.
The problem can be formulated on any manifold with a~notion
of differential whose tangent space carries an~action
of a~``model'' isometry group.

In this article, we study the geometric rigidity problem on the Heisenberg
groups $\mathbb{H}^n$, $n>1$. Here is the main result.

\begin{theorem}\label{th:stab_John}
Consider a~John domain $U$ with inner radius $\alpha$
and outer radius $\beta$ in the Heisenberg group
$\mathbb{H}^n$, $n>1$. Then, for every
$f\in I(1+\varepsilon, U)$
there exists an~isometry $\theta$ with
$$
\int_{U}\exp\Bigl(\Bigl(\frac{\beta}{\alpha}\Bigr)^{2n+3}\frac{N_1|D_hf(x)-D_h\theta(x)|}{\varepsilon}\Bigr)\,dx
\leqslant 16 |U|
$$
and
$$
\sup_{x\in U}d(f(x),\theta(x))\leqslant N_2 \frac{\beta^2}{\alpha} (\sqrt{\varepsilon}+\varepsilon).
$$
Here the constants $N_1$  and
 $N_2$ depend only on $n$.
\end{theorem}

Here $I(1+\varepsilon,U)$ is the class of quasi-isometries
(see Definition \ref{def:qi}), $D_h f(x)=\{X_if_j(x)\}_{i,j=1,\dots,2n}$
is the approximate horizontal differential, and
$d$ is the Carnot--Carath\'eodory metric.

The dilation $\delta_{1+\varepsilon}$ shows
that the proximity orders in Theorem \ref{th:stab_John}
are asymptotically sharp.

D.~Morbidelli and N.~Arcozzi \cite{morb-arc} investigated the geometric
rigidity problem for locally bi-Lipschitz mappings of the Heisenberg
group $\mathbb{H}^1$. We should note, however, that the proximity orders
($\varepsilon^{2^{-11}}$ in the uniform norm and
$\varepsilon^{2^{-12}}$ in the Sobolev norm) obtained in \cite{morb-arc}
are obviously far from being optimal.

Our proof of Theorem~\ref{th:stab_John}  develops Reshetnyak's approach  to the subject in the Euclidean case  \cite{resh}. The proof
essentially consists in linearizing the deformation tensor $E$
on Heisenberg groups as a~first-order differential operator with
constant coefficients whose kernel ``almost'' coincides with
the Lie algebra of the isometry group.

The most important motivation for the study of isometries in
sub-Riemannian geometry is given by the recently constructed
visualization model (see the papers by G.~Citti and A.~Sarti
\cite{citti-sarti} and R.~K.~Hladky and S.~D.~Pauls
\cite{hladky-puals}). The geometry of the model is based on the
roto-translation group, which is a~three-dimensional
non-nilpotent Lie group. However, it is a~contact manifold
whose tangent cone at each of point is the Heisenberg group
$\mathbb{H}^1$. The geometric rigidity problem finds an~unexpected
interpretation in sub-Riemannian geometry: a~local distortion of
an~image does not incur a~loss of global information about it.

In Section 2 we define quasi-isometries on Carnot--Carath\'eodory spaces,
introduce the main concepts used and prove
that the class of quasi-isometries under consideration
includes locally bi-Lipschitz mappings.
In Section 3 we introduce an~operator~$Q$ linearizing the strain tensor~$E$
on the Heisenberg group, and investigate its properties:
we describe its kernel and construct a~projection onto it.
In Section 4 we prove the geometric rigidity of isometries on the balls contained in a~given domain.
In Section 5 we prove Theorem~1 on a~John domain. There we also obtain a~partial extension of Theorem~1
to a~H\"older domain. In the Appendix we prove some auxiliary results.

The main results of this article were announced in \cite{vod-isan-dan}.

\section{Quasi-isometries}

\begin{definition}[cf. {\cite{G,KV,nsw}}]\label{carnotmanifold}
Fix a connected Romanian
$C^{\infty}$-mani\-fold~$\mathbb M$ of topological dimension~$N$. The
manifold~$\mathbb M$ is called a {\it Carnot--Carath\'{e}o\-dory space} if the
tangent bundle $T\mathbb M$ has a filtration
$$
H\mathbb M=H_1\mathbb M\subsetneq\ldots\subsetneq H_i\mathbb M\subsetneq\ldots\subsetneq H_M\mathbb M=T\mathbb M
$$
by subbundles such that each point $p\in\mathbb M$ has a neighborhood $U\subset\mathbb M$
equipped with a collection of $C^{1,\alpha}$-smooth vector fields
$X_1,\dots,X_N$, $\alpha\in(0,1]$, enjoying the following two properties. For each
$v\in U$,

$(1)$ $H_i\mathbb M(v)=H_i(v)=\operatorname{span}\{X_1(v),\dots,X_{\dim H_i}(v)\}$
is a subspace of $T_v\mathbb M$ of a constant dimension $\dim H_i$,
$i=1,\ldots,M$;

$(2)$ we have
\begin{equation}\label{tcomm}[X_i,X_j](v)=\sum\limits_{k:\,\operatorname{deg}
X_k\leq \operatorname{deg} X_i+\operatorname{deg}
X_j}c_{ijk}(v)X_k(v)
\end{equation}
where the {\it degree} $\deg X_k$ is defined as $\min\{m\mid X_k\in H_m\}$;

Moreover, if the third condition holds then the Carnot--Carath\'eodory space is called the {\it Carnot manifold}:

$(3)$ the quotient mapping $[\,\cdot ,\cdot\, ]_0:H_1\times
H_j/H_{j-1}\mapsto H_{j+1}/H_{j}$ induced by the Lie bracket is an
epimorphism for all $1\leq j<M$. Here $H_0=\{0\}$.

The subbundle $H\mathbb M$ is called {\it horizontal}.

The number $M$ is called the {\it depth} of the manifold $\mathbb
M$.
\end{definition}

The intrinsic {\it Carnot--Carath\'{e}odory distance} $d$ between
two points $x, y \in \mathbb{M}$ is defined as the infimum of
lengths of the horizontal curves joining $x$ and $y$ (a piecewise
smooth curve $\gamma$ is horizontal if $\overset{.}{\gamma}(t) \in
H\mathbb{M}(\gamma(t))$). This distance is correctly-defined \cite{KV}
and non-Riemannian if $n=\dim H\mathbb{M}\neq N$.

Let $U$ be a~domain in ${\Bbb M}$
and $\{X_1,\dots,X_n\}$ be an orthonormal basis of $H\mathbb M$ on $U$ from Definition~\ref{carnotmanifold}. {\it The Sobolev space}
$W^1_{q}(U)$, $1\le q\le\infty$,
consists of the  functions $f\colon U\to\Bbb R$
possessing the weak derivative $X_{i}f$ along the vector field
$X_{i}$, $i=1,\dots,n$, and having a finite norm
$$
\|f \|_{W^1_q(U)}=
\|f \|_{q,U}+
\|\nabla_{\mathcal{L}}f \|_{q,U},
$$
where
$\nabla_{\mathcal{L}}f=(X_{1}f,\dots,X_{n}f)$ is the
{\it subgradient} of~$f$ and $\|\cdot\|_{q,U}$ stands for the $L_q$-norm of a~measurable function on $U$. Recall that
a locally integrable function $g_i:U\to\Bbb R$ is called
the weak derivative of a~function $f$ along the vector field
$X_{i}$ if
$\int\limits_{U}g_i\psi\,dx=
-\int\limits_{U}fX_{i}\psi\,dx$ for every test function
$\psi\in C_0^{\infty}(U)$.

If $f\in W^1_q(U)$ for every bounded open set $U$, with $\overline U\subset\Omega$,
then $f$ is said to {\it be of class} $W^1_{q,\operatorname{loc}}(\Omega)$.

\begin{definition}
A mapping $f:\Omega\to {{\Bbb M}}$
{\it belongs to the Sobolev class} $W_{q,\operatorname{loc}}^1
(\Omega,{{\Bbb M}})$ if the following properties hold:

(A)~for each
$z\in {{\Bbb M}}$ the function
$[f]_z : x\in \Omega \mapsto d(f(x), z)$
belongs to $W_{q,\operatorname{loc}}^1 (\Omega)$;

(B)~the family of functions
$\{\nabla_\mathcal{L}[f]_z\}_{z\in {{\Bbb M}}}$
has a~majorant in $L_{q,\operatorname{loc}}(\Omega)$: there exists a~function
$g\in L_{q,\operatorname{loc}}(\Omega)$
independent of~$z$ such that
$|\nabla_\mathcal{L} [f]_z(x)|\le g(x)$
for almost all $x\in \Omega$.
\end{definition}

If $f$ is a Sobolev mapping then it can be redefined on a set of measure 
zero to be absolutely continuous on almost all lines of 
the horizontal vector fields. In this case 
there exist derivatives $X_if(x)$ a. e. in $\Omega$, moreover 
  $X_if(x)\in H_{f(x)}\mathbb M$, $i=1,\ldots, n$ (see \cite{pansu} in Carnot groups, and \cite{vod-07} in Carnot--Carath\'{e}o\-dory spaces).
A~transformation of the basis vectors $X_{i}(x)$,
$i=1,\dots,n$, of the horizontal subspace $H_x\mathbb{M}$ into
the horizontal vectors $(X_{i}f)(x)\in H_{f(x)}\mathbb{M}$
determines  a~mapping $D_h f(x)$ from the horizontal space $H_x\mathbb{M}$ into
$H_{f(x)}\mathbb{M}$ for almost all $x\in \Omega$, which is called
the {\it approximate horizontal differential}.

The mapping $D_h f$ in turn generates almost everywhere a~morphism $Df$ of graded Lie algebras \cite{vod-07}.
The determinant of the matrix $Df(x)$ is called the
{\it Jacobian} of~$f$ and is denoted by $J(x,f)$.

\begin{definition}\label{def:qi}
Let $U$ be an~open set in a~Carnot--Carath\'eodory space $\mathbb{M}$, and let
$f \colon U \to \mathbb{M}$
be a~nonconstant mapping of Sobolev class
$W^1_{1,\loc} (U,\mathbb{M})$.
The mapping $f$ belongs to the class $I(L,U)$, $L \geqslant 1$,
if $J(x,f)$ keeps its sign on~$U$
and
$
L^{-1}|\xi|\leqslant |D_h f(x) \xi|\leqslant L |\xi|
$
for all $\xi\in H \mathbb{M}(x)$ and almost every $x\in U$.
\end{definition}

Obviously, a~quasi-isometric mapping belongs to the Sobolev space
$W_{p,\loc}^1$ for all $p\geqslant 1$.

Recall that a~mapping $f\colon U\to\mathbb{M}$ is
{\it locally $L$-Lipschitz }
if every point
$x\in U$ has a~neighborhood $V$ with $\overline{V}\subset U$
such that  the inequality
$d(f(y),f(z))\leqslant L d(y,z)$ is valid for all $y,z\in V$;
also,
$f$ is {\it locally $L$-bi-Lipschitz}
if $\frac{1}{L}d(y,z)\leqslant d(f(y),f(z))\leqslant L d(y,z)$
for all $y,z\in V$.

\begin{lemma}\label{lem:Qiso->lip}
If $f$ belongs to $I(L,U)$ then $f$ is locally $L$-Lipschitz.
If in addition $f$ is a~local homeomorphism
then $f$ is locally $L$-bi-Lipschitz.

Conversely, every locally $L$-bi-Lipschitz mapping of an~open set
$U$ belongs to $I(L,U)$.
\end{lemma}

\begin{proof}
Since for every horizontal curve $\gamma\colon [0,T]\to U$ the curve
$f(\gamma)$ is also horizontal, it suffices to prove that
\begin{equation}
\label{eq:l(f(g))<Ll(g)}
l(f(\gamma))\leqslant L l(\gamma).
\end{equation}

If $f\circ \gamma \in ACL$ and  $D_h f(\gamma(t))$
is defined for almost all $t$ then \eqref{eq:l(f(g))<Ll(g)}
is obvious:
\begin{equation}
\label{eq:5}
l(f(\gamma))=\int_0^T \Bigl|\frac{d}{dt} f(\gamma(t))\Bigr|\,dt=
\int_0^T |D_h f(\gamma(t))|\Bigl|\frac{d}{dt} \gamma(t)\Bigr|\, dt
\leqslant
L l(\gamma).
\end{equation}

Take a~point $a\in U$ and a~field $X\in H\mathbb{M}$. Consider the curve $\gamma=\exp (t X)(a) $, $\gamma\colon  [0,T]\to U$,
and a~surface $S$ transversal to  $X$ at $a$
such that the foliation $\Phi=\{\exp (t X)(x),\, x\in S\}$
lies in~$U$. For the function $[f]_z(x)=d(z,f(x))$ there exists
a function $g\in L_1$ independent of~$z$ and satisfying both
$|[f]_z(x)-[f]_z(y)|\leqslant d(x,y) (g(x)+g(y))$
and $|\nabla_\mathcal{L} [f]_z(x)|\leqslant M g(x)$. By Fubini's theorem,
$g$ belongs to the class $L_1$ for almost all curves of the foliation
$\Phi$. Consequently,
$|[f]_z(x)-[f]_z(y)|\leqslant M\int_{[x,y]}g\,dt$ on each of these curves.
Choosing  $z$ arbitrarily close to $f(y)$, we infer that
$d(f(x),f(y))\leqslant M \int_{[x,y]}g\,dt$ and, hence,
$f\in ACL$ and is differentiable almost everywhere on almost all curves
in~$\Phi$. Consequently,
\eqref{eq:5} holds on almost all curves in $\Phi$. Choose a sequence of curves $\gamma_n\in\Phi$ converging to $\gamma$ and satisfying \eqref{eq:5}.
 Since~$f$ is continuous, $f\circ \gamma_n\to f\circ\gamma$ pointwise as $n\to \infty$. The lower semicontinuity of length yields 
 $$l(f\circ\gamma)\leqslant \underset{n\to\infty}{\operatorname{lim\,inf}}\,l(f\circ\gamma_n)\leqslant \underset{n\to\infty}{\operatorname{lim\,inf}}\, Ll(\gamma_n)=Ll(\gamma).$$ 
 Thus, the curve $\gamma=\exp(tX)(a)$ satisfies
\eqref{eq:l(f(g))<Ll(g)}.

Consider a domain $V$ with  $\overline{V}\subset U$.
Fix two points $x,y\in V$.
Then the points $x$ and $y$ can be joined by a~piecewise smooth horizontal curve
$\gamma$ in $U$ consisting of pieces of integral
curves of horizontal vector fields $X_i$, $i=1,\ldots,n$.
Moreover, $l(\gamma)\leqslant c\,d(x,y) $ with $c\geqslant 1$ dependent on $V$ \cite[Proof of Theorem 2.8.4]{KV};
hence, $l(f(\gamma))\leqslant c Ll(\gamma)$.
Thus, $f$ is locally Lipschitz with the Lipschitz constant $cL$ and $l(f\circ\gamma)\leqslant cLl(\gamma)$ for any horizontal curve in $V$.
Verify that $f$ is locally $L$-Lipschitz.

Suppose now that $\gamma\colon [0,T]\to V$ is a~horizontal curve parametrized by arc length.  Put $\Sigma=\{t\in[0,T]\mid  \gamma \text{ is not differentiable at } t\}$.
Then $|\Sigma|=0$, $|\overset{\cdot}\gamma(t)|=1$
for all $t\in[0,T]\setminus \Sigma$ and
$$
d(\gamma(t+s),\exp(s\overset{\cdot}{\gamma}(t))(\gamma(t)))=o(s)\quad
\text{as }s\to 0 \quad \text{for all } t\in [0,T]\setminus\Sigma.
$$

Consider arbitrary $\varepsilon>0$ and $\delta>0$. Now we construst a partition of the interval $[0,T)$ by intervals with diameter less than $\delta$.

First, we cover $\Sigma$ by open intervals $\{W\}_{W\in\mathcal W}$ centered at $[0,T]$ such that $|W|<\delta$, $\Sigma\subset \bigcup_{W\in\mathcal W}W$ and $\sum_{W\in\mathcal{W}}|W|<\varepsilon$.
Second, we cover set $[0,T]\setminus \Sigma$ by intervals
$\mathcal U=\{(t-\delta(t),t+\delta(t)):t\in [0,T]\setminus \Sigma\}$ where $\delta(t)>0$ satisfies
\begin{equation}\label{eq:uniform_curves}
d(\gamma(t+s),\exp(s\overset{\cdot}{\gamma}(t))(\gamma(t)))<\varepsilon s \quad
\text{for all }s\leqslant \delta(t).
\end{equation}
Without loss of generality we may assume that $\delta(t)<\delta/2$.

Since the set $[0,T]$ is compact, there is a finite covering of $[0,T]$ by open intervals $\{U_i\}$  with $U_i\in \mathcal U$ or $U_i\in\mathcal W$. By Lemma \ref{lem:partition}, there is a partition of $[0,T)$ by intervals $P_k=[t_k,t_{k+1})$ with the following properties:
$P_k\subset U_i$, for some $i$, and $\overline{P_k}$ contains the center of 
$U_i$. The latter we denote by $\tau_k$. Obviously, $t_{k+1}-t_k<\delta$. Divide indices into two groups:  $k\in I$ if $P_k\subset U_i$, $U_i\in\mathcal U$; and $k\in J$ if $P_k\in U_j$, $U_j\in\mathcal W$.

Since $\gamma$ is parameterized by arc length it follows that
$$
\sum\limits_{k\in J}d(f(\gamma(t_k)),f(\gamma(t_{k+1})))\leqslant
\sum\limits_{k\in J}l(f\circ \gamma|_{P_k})\leqslant cL\sum\limits_{k\in J}l(\gamma|_{P_k})
\leqslant cL\sum\limits_{W\in\mathcal W}|W| \leqslant cL\varepsilon.
$$

For $k\in I$ we set
$$
\sigma_k(t)=
\exp ((t-\tau_{k})\overset{\cdot}{\gamma}(\tau_k))(\gamma(\tau_{k})),\quad
t\in[t_k,t_{k+1}].
$$
Applying $|\overset{\cdot}{\sigma}_k(t)|=|\overset{\cdot}{\gamma}(\tau_k)|=1$
we obtain $l(\sigma_k)=t_{k+1}-t_k$ and $\sum\limits_{k\in I}l(\sigma_k)\leqslant T=l(\gamma)$.
Relation~\eqref{eq:uniform_curves} yields
$$
\sum\limits_{k\in I} d(\sigma_k(t_k),\gamma(t_{k}))+d(\sigma_k(t_{k+1}),\gamma(t_{k+1}))<\sum\limits_{k\in I} \varepsilon(t_{k+1}-t_{k}) \leqslant \varepsilon T.
$$
Therefore
\begin{multline*}
\sum\limits_{k\in I}d(f(\gamma(t_k)),f(\gamma(t_{k+1})))\\
\leqslant
\sum\limits_{k\in I}d(f(\gamma(t_k)),f(\sigma_k(t_{k})))+
d(f(\sigma_k(t_{k+1})),f(\gamma(t_{k+1})))
+l(f\circ\sigma_k)
\\
\leqslant
\sum\limits_{k\in I}cL\bigl(d(\gamma(t_k),\sigma_k(t_k))+d(\gamma(t_{k+1}),\sigma_k(t_{k+1})\bigr)+
Ll(\sigma_k) \leqslant cL\varepsilon T+Ll(\gamma).
\end{multline*}
Finally,
$$
l(f\circ\gamma)=\lim\limits_{\delta\to 0}
\sum\limits_{k\in I\cup J}d(f(\gamma(t_k)),f(\gamma(t_{k+1})))
\leqslant
Ll(\gamma)+
 cL\varepsilon T+cL\varepsilon.
 $$
Since $\varepsilon$ is arbitrary,  it follows that $l(f\circ\gamma)\leqslant Ll(\gamma)$.

The converse is obvious.
\end{proof}

The following partition lemma was used in the proof of Lemma \ref{lem:Qiso->lip}. Proof of this lemma is based on the induction method.

\begin{lemma}\label{lem:partition}
Consider the finite covering of a closed segment $[a,b]$ by open intervals $\{U_i\}$. Suppose
$x_i\in [a,b]$ where $x_i$ is centre of interval $U_i$.
Then there is a partition of $[a,b]$ by intervals $\{P_k\}$ satisfying
$x_i\in \overline{P_k}\subset U_i$ for some $i$.
\end{lemma}

\noindent \textsc{The Heisenberg Group.}
The Heisenberg group
$\mathbb{H}^n$ is an~example of homogeneous Carnot manifold. We may identify the points of $\mathbb H^n$ with
the points of $\mathbb{R}^{2n+1}$. The left-invariant vector fields
\begin{equation*}
X_i=\frac\partial{\partial x_i}
+2x_{i+n}\frac\partial{\partial x_{2n+1}},
\quad
X_{i+n}=\frac\partial{\partial x_{i+n}}
-2x_{i}\frac\partial{\partial x_{2n+1}},
\qquad
i=1,\dots, n,
\end{equation*}
constitute a~basis of the horizontal subbundle $H\mathbb{H}^n$.

Together with the vector field $X_{2n+1}=\frac\partial{\partial x_{2n+1}}$
they constitute the standard basis of the Lie algebra. The only nontrivial commutation relations are
$$
[X_j,X_{j+n}]=-4X_{2n+1}, \quad j=1,\dots,n.
$$

From now on we consider the Heisenberg group. It is convenient
to use the complex notation: a~point $x\in \mathbb{H}^n$ may be regarded
as $(z,t)$, where
$$z=(x_1+ix_{n{+}1},\ldots,x_n+ix_{2n})\in \mathbb{C}^n\quad\text{and}\quad
t=x_{2n+1}\in \mathbb{R}.$$ Then the vector fields
\begin{align*}
Z_j&=\frac{1}{2}(X_j-i\,X_{j+n})
   =\frac{\partial}{\partial
z_j}+i\overline{z}_j\frac{\partial}{\partial t},
   \quad
\overline{Z}_j=\frac{1}{2}(X_j+i\,X_{j+n})
   =\frac{\partial}{\partial \overline{z}_j}
   -i\,z_j \frac{\partial}{\partial t},
\quad j=1,\ldots,n,\\
T&=X_{2n+1}=\frac{\partial}{\partial t}
\end{align*}
constitute a~left-invariant basis of the Lie algebra.

The \textit{dilation} $\delta_s$, for $s>0$, acts on the Heisenberg group
as $\delta_s(z,t)=(sz,s^2t)$ and is an~automorphism
of it. The \textit{homogeneous norm}
$\rho(z,t)=(|z|^4+t^2)^{1/4}$ defines the \textit{Heisenberg metric}
$\rho$ as $\rho(x,y)=\rho(x^{-1}\cdot y)$,
$x,y\in \mathbb{H}^n$. Observe that the Heisenberg metric is a~metric and not
just a~quasi-metric:
$\rho(x\cdot y)\leqslant \rho(x)+\rho(y)$ for all $x,y \in \mathbb{H}^n$
(see \cite{isan-smj1} for instance). It is also known that
the Heisenberg metric $\rho$ and the Carnot--Carath\'eodory metric $d$
are equivalent: there exists a~constant $c>1$
such that $c^{-1}d(x,y)\leqslant \rho(x,y)\leqslant c d(x,y)$
for all $x,y\in \mathbb{H}^n$.

The Lebesgue measure $\mathbb{R}^{2n+1}$ is a~bi-invariant Haar measure.
For the ball $B(x,r)=\{y\in \mathbb{H}^n\colon \rho(x,y)<r\}$
we have $|B(x,r)|=r^{\nu}|B(0,1)|$, where
$\nu=2n+2$ is the \textit{homogeneous dimension} of the group
$\mathbb{H}^n$.

Consider a~Sobolev mapping~$f$. Since
$Df$ is a~homomorphism of graded Lie algebras,
it follows that for almost every $x\in \Omega$ there exists
a number $\lambda(x,f)$ such that
$$
Df(x)X_{2n+1}=\lambda(x,f) X_{2n+1}.
$$
Furthermore \cite{rei}, $\lambda(x,f)^n=\det D_h f(x)$
and $\lambda(x,f)^{n+1}=J(x,f)$.
In particular, $J(x,f)\geqslant 0$ almost everywhere on $\Omega$
for odd $n$.
Consequently, for odd $n$, there are no Sobolev mappings
changing the topological orientation. We now give the definition of
orientation introduced by  A.~Kor\'{a}nyi and H.~M.~Reimann in \cite{rei}.

\begin{definition}
A mapping $f$
of the Sobolev class
$W_{1,\loc}^1(\Omega,\Bbb{H}^n)$
{\it preserves} ({\it changes})
$KR$-{\it orienta\-tion} if
$\lambda(x,f)>0$ ($\lambda(x,f)<0$) for almost all
$x\in \Omega$.
\end{definition}

A mapping $f\in I(1,U)$ is  called  an~\textit{isometry} on $U$.
Every isometric mapping of the Heisenberg group $\mathbb{H}^n$
has the form $\pi_a\circ \varphi_A$ or $\iota\circ\pi_a\circ \varphi_A$, where
$\iota(z, t) =(\overline{z},-t)$ is a~reflection,
$\pi_a(x)=a \cdot x$ with $a\in \mathbb{H}^n$ is a~left translation,
$\varphi_A(x) = (Az,t)$ with $A \in U(n)$ is a~rotation \cite{rei}.
Isometries preserve distance in the Heisenberg metric
as well as in the Carnot--Carath\'eodory metric. It is also worth noting
that $D_h \varphi$ is a~constant mapping for every isometry~$\varphi$.

A quasi-isometric mapping is not only locally Lipschitz but also
a mapping with bounded distortion. Consider   a~domain $U$
in $\mathbb{H}^n$. Recall that a~nonconstant mapping
$f \colon U \to \mathbb{H}^n$
of the class $W^1_{\nu,\loc} (U,\mathbb{H}^n)$
is called a~{\it mapping with bounded distortion} if there exists
a constant $K \geqslant 1$ such that the approximate horizontal differential
satisfies
$|D_h f(x)|^\nu\leqslant K^{n+1} J(x,f)$
for almost every $x\in U$.
The smallest constant $K$ in this inequality is called
the ({\it linear}) {\it distortion coefficient} of $f$ and is denoted by $K(f)$.

Suppose that $f\in I(L,U)$. Denote by $\lambda_1$ and $\lambda_0$
 eigenvalues of
$D_h f(x)$ of the largest and smallest absolute values.
Clearly, $|\lambda_1|\leqslant L$,
$|\lambda_0|\geqslant L^{-1}$, and $|D_h f(x)|=|\lambda_1|$.
We also have $|J(x,f)|= |\lambda_1 \lambda_0|^{n+1}$. Hence,
$$|D_h f(x)|^{2n+2}=|\lambda_1|^{2n+2}=
\Bigl(\frac{|\lambda_1|}{|\lambda_0|}\Bigr)^{n+1}|J(x,f)|\leqslant
L^{2n+2}|J(x,f)|.$$
Thus, if $J(x,f)$ is nonnegative almost everywhere then $f$
is a~mapping of bounded distortion with $K(f)=L^2$.
If $J(x,f)$ is nonpositive almost everywhere then
$\iota\circ f$ is a~mapping with bounded distortion with $K(\iota\circ f)=L^2$.

\section{The Operator $Q$}

In this section we introduce a~differential operator $Q$
approximatign the equation $(D_hf(x))^t D_hf(x)=I$ to first order. This equation means
that $D_h f(x)$ is an~orthogonal matrix. In contrast to  the
Euclidean case, the horizontal differential of a Sobolev mapping
has some additional structure: up to
a factor, $D_h f(x)$ is a~symplectic matrix. Therefore, the operator~$Q$ consists
of two parts: the first is responsible for orthogonality, and the
second, for symplecticity.

\subsection{The main lemma for the operator $Q$}
Given a ~domain~$U$  in $\mathbb{H}^n$,
denote by $Q$  the homogeneous differential operator acting on a~mapping
$u \colon U \to \mathbb{R}^{2n}$ as
\begin{equation}
\label{eq:operator_Q}
Q u =
\frac{1}{2}\begin{pmatrix}
D_h u + (D_h u)^t\\
D_h u + J D_h u J
\end{pmatrix},
\quad
J=\begin{pmatrix}
0 & I\\
-I & 0
\end{pmatrix}.
\end{equation}
Here the $2n\times 2n$ matrix $D_h u$ equals $(X_i u_j)_{i,j=1,\ldots,2n}$.
The operator $Q$ also acts on mappings $u$ from $U$ to $\mathbb{H}^n$.
In this case, $D_h u$ in
\eqref{eq:operator_Q} stands for  the approximate horizontal
differential of $u$.

In complex notation, the operator $Q$ is defined as
$$
Qu=\begin{pmatrix}
\frac{1}{2}(Z u + (Zu)^*)\\
\overline{Z}u
\end{pmatrix},
\quad
u\colon U\to \mathbb{C}^n.
$$

The following lemma expresses the main inequality for the operator $Q$:
\begin{lemma}
\label{lem:main_estimate_Q}
Given   an~open set  $U$ in $\mathbb{H}^n$ and
a mapping $f$ of class $I(L,U)$
preserving $KR$-orientation,
the inequality
\begin{equation}\label{eq:main}
|Q (x^{-1}\cdot f(x))|\leqslant \frac{L^2-1}{2}\bigl(|D_h
f(x)-I|+2\bigr) +\frac{1}{2}|D_h f(x)-I|^2
\end{equation}
holds almost everywhere on $U$.
\end{lemma}

\begin{proof}
Put $x^{-1}\cdot f(x)=u(x)$.
Then $D_h f(x)=D_h u(x)+I$ for almost all
$x\in U$.
We have
$$
(D_h f(x))^t D_h f(x)=I+(D_h u(x))^t+D_h u(x)+
(D_h u(x))^t D_h u(x).
$$
Hence,
$$2 Q_1 u(x)=(D_h f(x))^t D_h f(x)-I-(D_h u(x))^t D_h u(x),$$
where
$Q_1 u(x)=\frac{1}{2}((D_h u(x))^t+D_h u(x))$
is a~first-order differential operator with constant
coefficients.
The relation $|(D_h f(x))^t D_h f(x)-I|\leqslant L^2-1$ yiels
$$|Q_1 u(x)|\leqslant \frac{L^2-1}{2}+\frac{1}{2}|D_h u(x)|^2.$$

It is easy to verify that
$|\overline{Z}f|=|\frac{1}{2}(D_h f +J D_h f J)|$
and $|Zf|=|\frac{1}{2}(D_h f -J D_h f J)|\leqslant |D_hf|$.
Since $f$ preserves $KR$-orientation and is a~mapping with bounded distortion,
the Beltrami system \cite[Theorem
C]{rei} implies that
$$
|\overline{Z}u|=
|\overline{Z}f|\leqslant \frac{K-1}{K+1} |Zf|\leqslant
\frac{K-1}{K+1} (|D_h f-I|+1)\leqslant \frac{L^2-1}{2} (|D_h u|+1).
$$

It remains to observe that
$|Qu|\leqslant |Q_1 u|+|\frac{1}{2}(D_h u +J D_h u J)|$.
\end{proof}

\subsection{The kernel of the operator $Q$}

To prove the main results of this paper, we apply the coercive estimates
for~$Q$ in~\eqref{eq:main}. On general Carnot groups,
Isangulova and Vodopyanov established coercive estimates
for homogeneous differential operators with constant coefficients
and finite-dimensional kernels \cite{isan-vod}. On Heisenberg groups, Romanovski{\u\i} obtained this result earlier in \cite{rom,rom1}. To apply the coercive estimates,
we only have to show that the kernel of $Q$ is finite-dimensional.

\begin{lemma}
\label{lem:kernel_Q}
The kernel of the operator $Q$ on the Sobolev class
$W_{p,\loc}^1(\mathbb{H}^n,\mathbb{C}^{n})$,
$p> 1$,
is finite-dimensional$:$
$u\in \kernel Q$ if and only if
\begin{align*}
u(z,t)&=a+Kz,
\quad \text{where }
a \in \mathbb{C}^{n} \ \text{and}\ K+K^*=0 \qquad \text{for }n>1;\\
u(z,t)&=a+ikz+tb+iz^2\overline{b}+i|z|^2b,
\quad \text{where }
a,b \in \mathbb{C},\ k\in \mathbb{R} \qquad \text{for }n=1;
\end{align*}
\end{lemma}

\begin{proof}
\noindent {\sc (i)} Take a~$C^\infty$-function
$u\colon \mathbb{H}^n \to \mathbb{C}^{n}$ in the kernel of
$Q$.
In complex notation, $u \in \kernel Q$ if and only if
$$
Zu=-(Zu)^*,
\quad
\overline{Z}u=Z\overline{u}=0.
$$

If $u$ is independent of $t=x_{2n+1}$ then
it is easy to see that
$$
u(z,t)=a+Kz, \quad
\text{where }
z\in \mathbb{C}^n, t\in\mathbb{C},
a\in \mathbb{C}^n, \text{ and }K+K^*=0.
$$

Suppose that $u$ depends on $t=x_{2n+1}$.
We have
\begin{equation*}
-2i Z_m Tu_k=  Z_m (Z_k\overline{Z}_k
-\overline{Z}_k Z_k) u_k
=-Z_m\overline{Z}_k Z_k u_k
=Z_m \overline{Z}_k \overline{Z}_k \overline{u}_k
=\overline{Z}_k \overline{Z}_k Z_m \overline{u}_k= 0
\end{equation*}
for all $m\neq k$. If $m=k$ then
$$
-2i Z_k Tu_k=  Z_k (Z_j\overline{Z}_j
-\overline{Z}_j Z_j) u_k
=-Z_k\overline{Z}_j Z_j u_k
=Z_k \overline{Z}_j \overline{Z}_k \overline{u}_j =
-2i \overline{Z}_j T\overline{u}_j=2i Z_j T u_j,
$$ where $j\neq k$.
Thus, $Z_k Tu_k=0$ for all $k=1,\ldots, n$ provided that
$n>2$.

\noindent {\sc (ii)}
Consider the case $n>2$. We have
$Tu=\lambda$ with
$\lambda\in \mathbb{C}^n$.
Verify that $\lambda=0$.
We have
$$
u=a+Kz+\lambda (t+i|z|^2)+P(z)
$$
with $a,\lambda\in \mathbb{C}^n$ and $K+K^*=0$.
Here $P=(P_1,P_2,\ldots,P_n)\colon \mathbb{C}^n\to \mathbb{C}^n$, where $P_k(z)=\sum_{l,s=1}^n p_{ls}^k z_lz_s$, $k=1,\ldots,n$,
are polynomials of degree 2 depending only on $z$,
$p^k_{ls}=p^k_{sl}$.
Here we consider the function $t+i |z|^2$ since
its differential  along
$\overline{Z}_k$ vanishes
for all $k=1,\ldots,n$.

Hence,
$$Z_k u_l=K_{lk}+2i\overline{z}_k \lambda_l+
\sum_{j=1}^n (p^l_{kj}z_j+p^l_{jk}z_j)=
K_{lk}+2i\overline{z}_k \lambda_l+
2\sum_{j=1}^n p^l_{kj}z_j,
$$
$$
\overline{Z_l u}_k=
\overline{K}_{kl}
-2i z_l \overline{\lambda}_k
+2\sum_{j=1}^n \overline{p}^k_{lj}\overline{z}_j.
$$
The coefficients of $\overline{z}_k$
and $z_l$ in the equation $Z_k u_l+\overline{Z_l u_k}=0$
are
$$
2i\lambda_l+2\overline{p}^k_{lk}=0,\quad
-2i \overline{\lambda}_k+2p^l_{kl}=0,\quad p^k_{ls}=0
\text{ for all } s\neq k.
$$
Since $p^k_{ls}=p^k_{sl}$, we infer that $P_k=0$ and $\lambda_k=0$
for all $k$.

\noindent \textsc{(iii)} Consider the case $n=2$.
We have
$$
Z_2Tu_1=Z_1Tu_2=0 \quad \text{and} \quad
Z_1 T u_1=-\overline{Z}_1 T\overline{u}_1=-Z_2Tu_2
=\overline{Z}_2T\overline{u}_2=\mu.
$$
The following relations show that $\mu$ is a~constant:
\begin{gather*}
Z_1 \mu=Z_1 \overline{Z}_2T\overline{u}_2=0,\qquad
\overline{Z}_1 \mu =-\overline{Z}_1Z_2Tu_2=0,\\
Z_2 \mu=-Z_2\overline{Z}_1 T\overline{u}_1=0,\qquad
\overline{Z}_2\mu=\overline{Z}_2Z_1Tu_1=0.
\end{gather*}
Hence, $Tu_1=\lambda_1+\mu z_1$ and $T u_2=\lambda_2-\mu z_2$.
Thus,
$$
u_1=(t+i|z|^2)(\lambda_1+\mu z_1)+P_1(z),\quad
u_2=(t+i|z|^2)(\lambda_2-\mu z_2)+P_2(z).
$$
Here we write down $u$ up to the known term $a+Kz$ and
$P_k=a_k z_1^2+b_k z_1z_2+c_kz_2^2$, $k=1,2$, are
polynomials of degree $2$ depending only on $z_1,z_2$.

It follows that
\begin{multline*}
0=Z_1 u_1+\overline{Z_1 u_1}=2i\overline{z}_1(\lambda_1+\mu z_1)+
(t+i|z|^2)\mu+2a_1z_1+b_1 z_2
\\
-2iz_1(\overline{\lambda}_1+
\overline{\mu} \overline{z}_1)+
(t-i|z|^2)\overline{\mu}+2\overline{a}_1\overline{z}_1+
\overline{b}_1 \overline{z}_2.
\end{multline*}
The coefficients of $|z|^2$ and $t$ are equal to
$i\mu-i\overline{\mu}$ and $\mu+\overline{\mu}$
respectively.
Thus, $\mu=0$.
Clearly, $b_1=0$ and $a_1=i\overline{\lambda}_1$.
Similarly, $b_2=0$ and $c_2=i\overline{\lambda}_1$.
The equality
$$
Z_2u_1=2i\overline{z}_2\lambda_1+2c_1z_2=
-\overline{Z_1u_2}
=-2i z_1\overline{\lambda}_2-2\overline{a}_2\overline{z}_1
$$
implies that $\lambda_1=\lambda_2=P_1=P_2=0$.

\noindent {\sc (iv)} Consider the case $n=1$.
A mapping $u=(u_1,u_2)\colon \mathbb{H}^1\to \mathbb{R}^2$
belongs to $\kernel Q$ if and only if
$$
X u_1=0, \quad Yu_2=0,\quad Y u_1+Xu_2=0.
$$
Put $\varphi=Y u_1=-Xu_2$. It satisfies
\begin{align*}
&X^2 \varphi=X^2 Yu_1=XXYu_1-XYXu_1=-4XTu_1=-4TXu_1=0,
\\
&Y^2 \varphi=-Y^2 Xu_2=YXYu_2-YYXu_2=-4YTu_2=-4TYu_2=0,
\\
&YX\varphi+XY\varphi=YXYu_1-XYXu_2
\\
& \qquad=Y(-4T+YX)u_1
+X(-4T-XY)u_2
 =-4T(Yu_1+Xu_2)=0.
\end{align*}
Verify that $T\varphi\equiv \mathrm{const}$.
We have
$$
-4XT\varphi=X(XY\varphi-YX\varphi)=-2XYX\varphi=(XY-YX)X\varphi=
XYX\varphi
$$
and
$$
-4YT\varphi=Y(XY\varphi-YX\varphi)=2YXY\varphi=
(XY-YX)Y\varphi=-YXY\varphi.
$$
Hence, $XYX\varphi=YXY\varphi=0$ and $XT\varphi=YT\varphi=0$.
Thus, $T\varphi=\lambda\in \mathbb{R}$
and $\varphi=\lambda t+\psi(x,y)$. Since
$X^2\varphi=\frac{\partial^2\psi}{\partial x^2}=0$,
$Y^2 \varphi=\frac{\partial^2\psi}{\partial y^2}=0$, and
$XY\varphi+YX\varphi=2\frac{\partial^2\psi}{\partial y\partial x}=0$,
we conclude that $\psi$ is a linear function of $x$ and $y$.

Thus, $\varphi=\alpha+\lambda t+\mu x+\nu y$.
It remains to calculate $u_1$ and $u_2$. The systems
$$
\begin{cases}
Xu_1=0,\\
Yu_1=\varphi,\\
\end{cases}
\quad
\begin{cases}
Xu_2=-\varphi,\\
Yu_1=0\\
\end{cases}
$$
yield
$$
u_1=c_1+\alpha y-\frac{\mu}{4}t+\frac{\mu xy+\nu y^2}{2},\quad
u_2=c_2-\alpha x-\frac{\nu}{4}t-\frac{\nu xy+\mu x^2}{2}.
$$

\noindent {\sc (v)} Consider a~mapping $u$ of Sobolev class
$W_{p,\loc}^1(\mathbb{H}^n,\mathbb{R}^{2n})$
satisfying $Qu=0$ in the sense of distributions.
We show that $u\in \kernel Q$, where $\kernel Q$ is the finite-dimensional space
found in the smooth case.
Consider a~ball $B$ in $\mathbb{H}^n$
and construct a~sequence
$u_k\in C^\infty(\mathbb{H}^n,\mathbb{R}^{2n})$
such that $\|u-u_k\|_{W_{p}^1(B)}\to 0$ as $k\to \infty$.
We showed above that the kernel of $Q$ is finite-dimensional on smooth mappings.
Hence, by Theorem~1 of~\cite{isan-vod}, there exists a~projection $P$
onto $\kernel Q$
such that
$$
\|u_k-Pu_k\|_{W_{p}^1(B)}\leqslant C \|Qu_k\|_{p,B}.
$$
Passing to the limit as $k\to \infty$,
we infer that $\|u-Pu\|_{W_{p}^1(B)}\leqslant C \|Qu\|_{p,B}=0$,
where $Pu=\lim_{k\to \infty} Pu_k$. Since $Pu_k\in \kernel Q$, it follows that
$Pu$ also belongs to $\kernel Q$. Finally, $u=Pu\in \kernel Q$.
\end{proof}

\subsection{Projection onto the Kernel of the Operator $Q$}

In this subsection, we construct a~projection onto $\kernel Q$
convenient for further calculations.

Put
$$\mathrm{Box}(a,r)=\{ay\in \mathbb{H}^n\ :\  y=(y_1,\ldots,y_{2n+1}),
|y_i|< r,\ i=1,\ldots,2n,\ |y_{2n+1}|<r^2\}.$$
It is easy to verify that
$$\mathrm{Box}(a,\varkappa r)\subset B(a,r)\subset \mathrm{Box}(a,r),
\quad
\text{where } \varkappa=(4n^2+1)^{-1/4},$$
$$
|\mathrm{Box}(a,r)|=2^{2n+1}r^\nu,
\qquad
\int_{\mathrm{Box}(0,r)} |z_i|^2\, dx=
\frac{2^\nu r^{\nu+2}}{3}
\quad \text{for all}\quad i=1,\ldots,n.
$$

By \cite{isan-vod},
we have the following result:
\textit{given a~ball $B\subset \mathbb{H}^n$, $n>1$, and $p>1$
there is a~projection $\Pi$ from $W_p^1(B, \mathbb{R}^{2n})$ onto the kernel of $Q$
such that}
$$\|f-\Pi f\|_{W_{p}^1(B)}\leqslant C \|Qf\|_{p,B}
\quad
\text{\it for every } f\in W_p^1(B, \mathbb{R}^{2n}).
$$
By analogy with Theorem 3.2 of Chapter~3 of \cite{resh},
we can show that the coercive estimates
hold for every projection onto the kernel of $Q$.

\begin{proposition}[\mbox{\cite[Proposition 2]{isan-smj2}}]
Consider a~ball~$B$  on the Heisenberg group
$ \Bbb{H}^n$, $n>1$, $p>1$, and a~projection~$P$ from
$W_p^1(B,\Bbb{R}^{2n})$ onto $\operatorname{ker} (Q)$.
Then there is a~constant $C>0$ such that
$$
\|u-P u\|_{W_p^1(B) }\le C\|Qu\|_{p,B }
$$
for every $u\in W_p^1(B ,\Bbb{R}^{2n})$.
\end{proposition}

We construct a~projection $P$ from
$W_p^1(B(0,\frac{3}{10}),\mathbb{C}^{n})$ for $B(0,\frac{3}{10})\subset \mathbb{H}^n$ with
$n>1$ and $p>1$,
onto the kernel of $Q$.

Consider the complex-valued $n\times n$ matrix $A(u)$,
$$
[A(u)]_{ij}=\frac{2^{\nu+4} 3}{\varkappa^{\nu+2}}
\int_{\mathrm{Box}(0,\frac{\varkappa}{4})} u_i(x) \overline{z}_j \, dx,
\quad
i,j=1,\ldots,n;
$$
and the vector $a(u)\in \mathbb{C}^n$,
$$
[a(u)]_i=\frac{2^{\nu+1}}{\varkappa^\nu}
\int_{\mathrm{Box}(0,\frac{\varkappa}{4})} u_i(x) \, dx,
\quad
i=1,\ldots,n.
$$

The following properties are obvious:

(1) if
$u\equiv \mathrm{const}$ then $a(u)=u$
and $A(u)=0$;

(2) if $u(z,t)\equiv z$ then $a(u)=0$
and $A(u)=I$;

(3) if $B\in U(n)$ then $a(B u)=Ba(u)$
and $A(B  u)=B A(u)$.

\begin{definition}
Define the projection $P$ onto the kernel of $Q$ as
$$Pu=K(u)z+a(u)$$
for
$u\in W_p^1(B(0,\frac{3}{10}),\mathbb{C}^{n})$,
where $K(u)=\frac{A(u)-[A(u)]^*}{2}$ is a~skew-Hermitian $n\times n$ matrix.
\end{definition}

\begin{lemma}\label{lem:P=0}
Suppose that $\varepsilon<\sqrt{\frac{2}{n}}
\bigl(\frac{2}{\varkappa}\bigr)^{n+1}$. If $u\in W_2^1(B(0,\frac{3}{10}),\mathbb{C}^n)$
satisfies $|u(x)-z|\leqslant \varepsilon$
for all $x=(z,t)\in B(0,\frac{3}{10})$
then there is a~unitary $n\times n$ matrix $V\in U(n)$ such that
$|V-I|< \frac{n \varkappa^{n+1}}{2^n}\varepsilon$
and $P(Vu)\equiv \mathrm{const}$.
\end{lemma}

\begin{proof}
Put $A=A(u)$.
Given a~vector $\xi\in \mathbb{C}^n$, we have
\begin{multline*}
\bigl|[A-I]\xi\bigr|^2=
\bigl|[A(u)-A(z)]\xi\bigr|^2=
\frac{2^{\nu+4}3}{\varkappa^{\nu+2}}
\sum_{i=1}^n\,
\biggl|
\int\limits_{\mathrm{Box}(0,\frac{\varkappa}{4})}
\sum_{j=1}^n
(u_i(x)-z_i) \, \overline{z}_j \, \xi_j \, dx
\biggr|^2
\\
\leqslant
\frac{2^{\nu+4}3}{\varkappa^{\nu+2}}
\int\limits_{\mathrm{Box}(0,\frac{\varkappa}{4})}
|u(x)-z|^2\,dx
\int\limits_{\mathrm{Box}(0,\frac{\varkappa}{4})}
|\langle \xi,z \rangle |^2
\leqslant
\frac{n \varkappa^\nu}{2^{\nu+1}}\varepsilon^2|\xi|^2.
\end{multline*}
Hence, $|A-I|\leqslant \sqrt{\frac{n}{2}}
\bigl(\frac{\varkappa}{2}\bigr)^{n+1} \varepsilon$
and $A$ is a~nondegenerate complex
$n\times n$ matrix if
$\varepsilon<\sqrt{\frac{2}{n}} \bigl(\frac{2}{\varkappa}\bigr)^{n+1}$.

For the positive definite Hermitian
$n\times n$ matrix $A^* A$, there exists a~unitary matrix
$U\in U(n)$ such that $UA^*AU^*$ is a~real diagonal matrix
$\operatorname{diag}\{\mu_1,\ldots,\mu_n\}$, $\mu_i>0$ (for instance, see
\cite{malcev}).
Hence, there are two orthonormal bases
$\{w_i=U^* e_i\}_{i=1,\ldots,n}$
and $\{v_i=\frac{1}{\lambda_i} Aw_i\}_{i=1,\ldots,n}$, where
$\lambda_i=\sqrt{\mu_i}>0$ and $Aw_i=\lambda_i v_i$ for $i=1,\ldots,n$.
Here $\{e_i=(0,\dots,0,\overset{i}{1},0,\dots,0)\}_{i=1,\dots,n}$
is the standard basis of $\mathbb{C}^n$.

Consider  the unitary matrix $V\in U(n)$ with
$Vv_i=w_i$ for $i=1,\ldots,n$.
Since $VA w_i=V(\lambda_i v_i)=\lambda_i w_i$,
the matrix $VA$ is diagonal in the basis $\{w_1,\ldots,w_n\}$, and
hence, Hermitian in the origin basis $\{e_i\}_{i=1,\dots,n}$.
Therefore,
$A(V u)=V A(u)$ is a~Hermitian matrix, and consequently,
$K(V u)=0$.
Thus, we have demonstrated that
$P(Vu)=a(Vu)\equiv \mathrm{const}$.

Estimate $|V-I|$.
Since
$|Aw_i-w_i|=|\lambda_iv_i-w_i|\leqslant \sqrt{\frac{n}{2}}
\bigl(\frac{\varkappa}{2}\bigr)^{n+1} \varepsilon$
for all $i=1,\ldots,n$, we obtain
$$
|\lambda_i-1|=\bigl||Aw_i|-|w_i|\bigr|\leqslant |Aw_i-w_i|
\leqslant \sqrt{\frac{n}{2}}
\biggl(\frac{\varkappa}{2}\biggr)^{n+1} \varepsilon
$$
and
$$
|v_i-w_i|\leqslant |\lambda v_i-w_i|+|\lambda_i v_i-v_i|\leqslant
\sqrt{2n}
\biggl(\frac{\varkappa}{2}\biggr)^{n+1} \varepsilon.
$$
Given a~vector $\xi=\sum_{i=1}^n \xi_i v_i\in \mathbb{C}^n$,
we have
$$
\bigr|(V-I)\xi\bigl|
=
\biggl|\sum_{i=1}^n \xi_i w_i-\xi_i v_i\biggr|
\leqslant
|\xi|\sqrt{\sum_{i=1}^n |w_i-v_i|^2}
\leqslant
\sqrt{2}|\xi|^2  n^2 \biggl(\frac{\varkappa}{2}\biggr)^{n+1} \varepsilon.$$
Hence,
$|V-I|< \frac{n \varkappa^{n+1}}{2^n}\varepsilon$.
\end{proof}

\section{Local Geometric Rigidity}

\subsection{Qualitative local rigidity}

\begin{lemma}
\label{th:Local_Sob_stab}
For every
$q\in (0,1)$,
there exist nondecreasing functions
$\mu_i(\cdot,q)\colon [0,\infty) \to [0,\infty)$,
$i=1,2$,
such that

{\rm(1)} $\mu_i(t,q)\to 0$ as $t\to 0$,
$i=1,2${\rm;}

{\rm(2)} for each mapping
$f$ of class $I(1+\varepsilon,B(0,1))$, where $B(0,1)\subset \mathbb{H}^n$,
there exists an~isometry~$\theta$
satisfying
\begin{align*}
\rho( f(x),\theta(x))
&\leqslant q\,\mu_1(\varepsilon ,q)
\quad
\text{for all }x\in B(0,q),
\\
\|D_h f(x)-
D_h\theta(x)\|_{2,B(0,q)}
&\leqslant
|B(0,q)|^{1/2} \mu_2 (\varepsilon,q).
\end{align*}
\end{lemma}

\begin{proof}
Put $B=B(0,1)$,
$$
\mu_1(\varepsilon,q)=\frac{1}{q}
\sup_{f\in I(1+\varepsilon,B)}
\inf \biggl\{ \sup_{x\in B(0,q)}\rho(f(x),\varphi(x)):
\varphi \text{ is an~isometry} \biggr\}
$$
and
$$
\mu_2(\varepsilon,q)=\sup_{f\in I(1+\varepsilon,B)}
\inf \biggl\{ \frac{\|D_h f-D_h \varphi\|_{2, B(0,q)}}{|B(0,q)|^{1/2}}:
\varphi \text{ is an~isometry} \biggr\}.
$$
Property (2) is obvious.

It remains to prove that $\mu_1$ and $\mu_2$
enjoy property (1).

\textsc{(i)}
Assume that for some $q\in (0,1)$
the function $\mu_1(t,q)$ fails to tend to $0$ as $t\to 0$.
Then there exist $\delta>0$
and a~sequence of quasi-isometries
$\{f_j\in I(L_j,B)\}$ with $L_j<1+\frac{1}{j}$
such that
\begin{equation}
\label{eq:ot_protivnogo}
\sup_{x\in B(0,q)} \rho(f_j(x),\varphi(x))\geqslant
\varepsilon
\quad \text{for all }j\in \mathbb{N}
\end{equation}
for every isometry~$\varphi$.
Since the isometry group contains translations and reflections, we may
assume that $f_j(0)=0$ and $J(x,f)>0$ almost everywhere on $B(0,1)$.
By Lemma \ref{lem:Qiso->lip}, the sequence $\{f_j\}$ is
an equicontinuous and uniformly bounded family on every domain
compactly embedded into $B(0,1)$, for example, on the ball $B(0,q)$.
Consequently, there exists a~mapping $f_0\colon B(0,q)\to \mathbb{H}^n$
and a~subsequence uniformly converging to $f_0$, which we also
denote by $\{f_j\}$. Since all quasi-isometric mappings are
of bounded distortion, by \cite{vod-02}
$f_0$ is a~mapping with 1-bounded distortion. Verify that
$f_0$ is an~isometry.

The weak convergence of the Jacobians \cite{vod-02} yields
$$
\lim_{j \to \infty}\int_B J(x,f_j) \,\xi(x)\, dx=
\int_B J(x,f_0) \,\xi(x)\, dx
$$
for every $\xi\in C_O(B)$. On the other hand,
$$
L_j^{-\nu}\int_B \xi(x)\, dx\leqslant\int_B J(x,f_j) \,\xi(x) \,dx\leqslant
L_j^\nu\int_B \xi(x)\, dx.$$
Consequently, $J(x,f_0)\equiv 1$ almost everywhere on $B(0,q)$.
This is possible only if $f_0$ is an~isometry.
Applying \eqref{eq:ot_protivnogo} for $f_0$, we arrive at a~contradiction.

\noindent \textsc{(ii)} Now we prove property (1) for $\mu_2$.
Assume the contrary. Then there exist numbers $\varepsilon>0$, $q \in (0,1)$,
a sequence of quasi-isometric mappings
$\{f_j\in I(L_j,B)\}$ with $L_j<1+\frac{1}{j}$,
and a~sequence of isometries $\theta_j$ such that
$$
\sup_{x\in B(0,q)}\rho(f_j(x),\theta_j(x))
\leqslant \mu_1(L_j-1,q)
\quad
\text{and}
\quad
\|D_h f_j(x)- D_h \theta_j(x)\|_{2,B(0, q)} dx\geqslant
\varepsilon.
$$
Like in part \textsc{(i)} of the proof, we may assume that
the sequence $\{f_j\}$ converges to an~isometry $f_0$ uniformly
on the ball $B(0,q)$. Clearly, the mappings $\theta_j$ converge to
$f_0$ uniformly on $B(0,q)$ as $j\to \infty$.
Therefore, $|D_h \theta_j(x)-D_h f_0(x)|\to 0$ and $|D_h f_j(x)|\to 1=|D_h f_0(x)|$
as $j\to \infty$ for all $x\in B(0,q)$.
Since the space $W_{2}^1$ is uniformly convex, the convergence of the norms
along with the uniform convergence $f_j\to f_0$ imply the convergence
$
\int_{B(0,q)}|D_h f_j(x)-D_h f_0(x)|^2 dx\to 0
$.
The properties of uniformly convex spaces can be found in
\cite{dan-swa}.
We arrive at a~contradiction:
\begin{multline*}
\varepsilon\leqslant\|D_h f_j(x)- D_h \theta_j(x)\|_{2,B(0, q)}
\\
\leqslant
\|D_h f_j(x)- D_h f_0(x)\|_{2,B(0, q)}+
\|D_h \theta_j(x)- D_h f_0(x)\|_{2,B(0, q)}\underset{j\to \infty}{\to} 0.
\end{multline*}
\end{proof}

\subsection{Application of the operator $Q$}

In this section we apply the coercive estimate
for the operator $Q$.
In view of the connection between the
Lie algebra of isometries and the kernel of $Q$,
we obtain the following lemma,
which shows that we can slightly perturb
the isometry of Lemma \ref{th:Local_Sob_stab}
to make the projection vanish.

\begin{lemma}
\label{the:coer_Q}
Take $n>1$.
There exist constants
$c_1=c_1(n)>0$ and $\varepsilon_1=\varepsilon_1(n)>0$
and a~nondecreasing function
$\mu_3\colon [0,\varepsilon_1)\to [0,\infty)$ with
$\mu_3(t)\to 0$ as $t\to 0$
such that,
given a ball $B(a,r)\subset \mathbb{H}^n$
and a mapping
$f\in I(1+\varepsilon, B(a,r))$ with
$\varepsilon<\varepsilon_1$,
there is an~isometry~$\theta$ satisfying
\begin{align*}
\|D_h f-D_h \theta\|_{2,B(a,\frac{3r}{10})}
&\leqslant c_1 \|Q ( x^{-1}\cdot (\theta^{-1}\circ f)(x))\|_{2,B(a,\frac{3r}{10})},
\\
\|D_h f-D_h \theta\|_{2,B(a,\frac{r}{2})}
&\leqslant
\Bigl|B\Bigl(a,\frac{r}{2}\Bigr)\Bigr|^{1/2}
\mu_3(\varepsilon).
\end{align*}
Here
$Q$ is the differential operator \eqref{eq:operator_Q}.
\end{lemma}

\begin{proof}
Assume first  that $B(a,r)=B(0,1)$.
Consider a~mapping $f\in I(1+\varepsilon,B(0,1))$.
By Lemma \ref{th:Local_Sob_stab},
there exists an~isometry~$\varphi$
such that the mapping
$g=\varphi^{-1}\circ f\in I(1+\varepsilon,B(0,1))$
satisfies
$$
|\widetilde{g}(x)-z|\leqslant
\rho(g(x),x)=
\rho(f(x),\varphi(x))\leqslant \frac{\mu_1(\varepsilon,1/2)}{2}
$$
for all $x=(z,t)\in B(0,1/2)$
and
$$
\int_{B(0,1/2)}|D_h g(x)-I|^2 dx=
\int_{B(0,1/2)}|D_h f(x)-D_h\varphi(x)|^2 dx\leqslant (\mu_2(\varepsilon,1/2))^2
|B(0,1/2)|.
$$
(Here $\widetilde{g}$ stands for the projection of  $g$ onto the first $n$
complex coordinates.)

Take $\varepsilon<\varepsilon_1$, where
$\frac{\mu_1(\varepsilon_1,1/2)}{2} \leqslant \sqrt{\frac{2}{n}}
\bigl(\frac{2}{\varkappa}\bigr)^{n+1}$.
By Lemma~\ref{lem:P=0}, there exists a~matrix $V\in U(n)$
such that
$|V-I|<\frac{n\varkappa^{n+1}}{2^{n+1}}\mu_1(\varepsilon,1/2)$
and
$D_h P(V\widetilde{g})\equiv 0$.

Put $\theta^{-1}=\varphi_V\circ\varphi^{-1}$.
We have $D_h P(\widetilde{x^{-1}\cdot (\theta^{-1}\circ f)(x)})=
D_h P(\widetilde{x^{-1}})+D_h P(V\widetilde{g})\equiv 0$.
By the coercive estimate \cite{isan-vod}, there is a~constant $c_1=c_1(n)>0$
such that
$$
\|D_h f-D_h \theta\|_{2,B(0,3/10)}=
\|D_h (\theta^{-1}\circ f)-I\|_{2,B(0,3/10)}
\leqslant
c_1\|Q(x^{-1}\cdot (\theta^{-1}\circ f)(x))\|_{2,B(0,3/10)}.
$$

We have
$$
|D_h f-D_h \theta|
=
|V D_h g-I|\leqslant
|D_h g-I |+|V-I|.
$$
Hence,
$$
\frac{\|D_h f-D_h \theta\|_{2,B(0,1/2)}}{|B(0,1/2)|^{1/2}}
\leqslant
\mu_2(\varepsilon,1/2)+\frac{n\varkappa^{n+1}}{2^{n+1}}\mu_1(\varepsilon,1/2)
=\mu_3(\varepsilon).
$$

To complete the proof, consider an~arbitrary ball $B(a,r)$ and
a mapping $f$ of class \linebreak $I(1+\varepsilon, B(a,r))$.
Then the mapping $g=\delta_{\frac{1}{r}}\circ
\pi_{-a}\circ f\circ \pi_{a}\circ \delta_{r}$
belongs to the class $I(1+\varepsilon,B(0,1))$.
Hence, there is an~isometry $\psi$ close to
$g$ satisfying the estimates of the lemma.
Then
$\theta=\pi_a\circ \delta_r\circ \psi\circ \delta_{\frac{1}{r}}\circ \pi_{-a}$
is a~required isometry for $f$.
\end{proof}

\subsection{Quantitative local rigidity}
\begin{lemma}
\label{the:qualitative_local}
Tale $n>1$.
Given a ball $B(a,r)\subset \mathbb{H}^n$
and a mapping $f\in I(1+\varepsilon, B(a,r))$,
there is an~isometry $\varphi$ satisfying
\begin{align*}
\frac{1}{|B(a,\frac{3r}{10})|}\int_{B(a,\frac{3r}{10})}|D_h f(x)-D_h \varphi(x)|^2\,dx
&\leqslant
(c_2\varepsilon)^2.
\end{align*}
The constant $c_2$ depends only on $n$.
\end{lemma}

\begin{proof}
Put $B_1=B(a,\frac{3r}{10})$, $B_2=B(a,\frac{r}{3})$, $B_3=B(a,\frac{r}{2})$.
By Lemma \ref{the:coer_Q}, there is an~isometry $\theta$ such that
\begin{align*}
\|D_h(\theta^{-1}\circ f)-I\|_{2,B_1}
&\leqslant c_1 \|Q (x^{-1}\cdot (\theta^{-1}\circ f)(x))\|_{2,B_1},
\\
\|D_h(\theta^{-1}\circ f)-I\|_{2,B_3}
&\leqslant
|B_3|^{1/2} \mu_3(\varepsilon).
\end{align*}

Put $g=\theta^{-1}\circ f\in I(1+\varepsilon, B(a,r))$.
By Proposition~4 of \cite{isan-smj2},
there is a~number $\varepsilon_2>0$
such that $g$ preserves $KR$-orientation on $B_1$ if
$\varepsilon<\varepsilon_2$.
Thus, assuming that $\varepsilon<\min\{\varepsilon_1,\varepsilon_2\}$,
we may apply Lemmas \ref{lem:main_estimate_Q} and \ref{the:coer_Q}. We obtain
\begin{equation*}
\|D_hg-I\|_{2,B_1}
\leqslant
c_1\frac{\varepsilon^2+2\varepsilon}{2}
\bigl(\|D_hg-I\|_{2,B_1}+2
|B_1|^{1/2}\bigr) +
\frac{c_1}{2}\|D_h g-I\|_{4,B_1}^2.
\end{equation*}

Estimate $\int_{B_1}|D_h g(x)-I|^4dx$.
Consider an~arbitrary  ball $B=B(a_0,r_0)\subset B_2$.
Then $B(a_0,2r_0)\subset B(a,r)$
and, by Lemma \ref{the:coer_Q},
there is an~isometry $\theta_B$,
and hence a~matrix
$A_B=D_h \theta_B\in U(n)$ such that
$$
\int_B|D_hg(x)-A_B|^2\,dx\leqslant (\mu_3(\varepsilon))^2|B|
=(\mu_3(\varepsilon))^2\int_B|A_B|^2\,dx.
$$
For the ball $B_2$, we have $A_{B_2}=I$
and
$$
\int_{B_2}|D_h g(x)-I|^2\,dx\leqslant
\int_{B_3}|D_h g(x)-I|^2\,dx
\leqslant
|B_2| (3/2)^\nu(\mu_3(\varepsilon))^2.
$$
It follows that
$D_h g$ is a~mapping with bounded specific oscillation
in the sense of $L_2$ relative to the class
$U(n)$ on the ball $B_2$ (see \cite[Definitions 1 and 2]{isan-smj1}).
Thus, by the Corollary to Theorem 1 of \cite{isan-smj1},
there are constants $C,\sigma_0>0$
such that on the ball $B_1=\frac{9}{10}B_2$ we have
$$
\int_{B_1}|D_h g(x)-I|^{4} dx
\leqslant C (3/2)^\nu(\mu_3(\varepsilon))^2
\int_{B_1}|D_h g(x)-I|^2\, dx
$$
provided that $(3/2)^{n+1}\mu_3(\varepsilon)<\frac{\sigma_0}{4}$.
We need to consider $\varepsilon<\varepsilon_3$,
where $\varepsilon_3\leqslant \min\{\varepsilon_1,\varepsilon_2\}$
and $(3/2)^{n+1}\mu_3(\varepsilon_3)\leqslant \frac{\sigma_0}{4}$.

Finally,
\begin{equation*}
\|D_hg-I\|_{2,B_1}
\leqslant
\frac{c_1\varepsilon(\varepsilon+2)}{2}
\bigl(\|D_hg-I\|_{2,B_1}+2
|B_1|^{1/2}\bigr) +
\frac{c_1\sqrt{C}3^{n+1}\mu_3(\varepsilon)}{2^{n+2}} \|D_h g-I\|_{2,B_1}.
\end{equation*}
Take $\varepsilon_4\leqslant \varepsilon_3$
such that
$$
\frac{c_1\sqrt{C} 3^{n+1}\mu_3(\varepsilon_4)}{2^{n+2}}
\leqslant\frac{1}{4},\qquad
\frac{c_1\varepsilon_4(\varepsilon_4+2)}{2}\leqslant\frac{1}{4}.
$$
If
 $\varepsilon< \varepsilon_4$
then
$$
\|D_h g -I\|_{2,B_1}=
\|D_h f-D_h \theta\|_{2,B_1}
\leqslant
2c_1(\varepsilon_4+2)\varepsilon|B_1|^{1/2}.
$$

Thus, we have established the lemma for $f\in I(1+\varepsilon,B(a,r))$
with $\varepsilon<\varepsilon_4$.
In the case
$\varepsilon\geqslant \varepsilon_4$, given an~isometry $\varphi$, we obviously have
$$
\frac{1}{|B_1|
}\int_{B_1}|D_h f(x)-D_h\varphi(x)|^2\,dx\leqslant
(2+\varepsilon)^2\leqslant \Bigr(\frac{2}{\varepsilon_4} +1\Bigl)^2\varepsilon^2.
$$
The lemma is proved with the constant
$c_2=\max\{2c_1(\varepsilon_4+2),\frac{2}{\varepsilon_4}+1\}$.
\end{proof}

\section{Global Geometric Rigidity}

In this section we prove Theorem \ref{th:stab_John}.
Local rigidity (Lemma \ref{the:qualitative_local}) means, in particular,
that the horizontal differential of a~quasi-isometry is a~$BMO$ mapping.
To pass from local rigidity to global rigidity, we apply the
John--Nirenberg technique.
In the Euclidean case, a~necessary and sufficient condition
for the exponential integrability of a $BMO$ mapping
 is that $U$ is a~H\"older domain \cite{smith-ste,hurri}.
In the metric space setting, this also holds
(see \cite{buck} for instance).
Thus, we can prove global geometric rigidity
in the Sobolev norm on H\"older domains.
Note that we can prove geometric rigidity in the uniform norm
only on John domains.

To begin with, we give definitions of John and H\"older domains and
some of their properties on
a metric space $(\mathbb{X},\rho)$.
For a~domain $U\subset \mathbb{X}$,
denote the distance from a~point $x\in U$ to the boundary $\partial U$
by $\rho_U(x)=\dist(x,\partial U)$.
For a~ball $B\subset \mathbb{X}$, denote by $x(B)$ and $r(B)$
its center and radius respectively.

\begin{definition}[\cite{john,buck_boman}]\label{def:John}
A bounded open proper subset $U$ of a~metric space
$(\mathbb{X}, \rho)$ with a~distinguished
point $x_*\in U$ is called a
(\textit{metric}) \textit{John domain} if it satisfies the
following ``twisted cone'' condition:
there exist constants $\beta\geqslant \alpha  > 0$ such that for all
$x \in U$ there is a~curve
$\gamma: [0, l] \to U$ with $l\leqslant \beta$
parameterized by arc length such that
$\gamma(0) = x$,
$\gamma(l) = x_*$ and
$\rho_U(\gamma(s)) \geqslant \frac{\alpha}{l}s$.

The numbers $\alpha$ and $\beta$ are the \textit{inner} and \textit{outer radii}
of $U$.
\end{definition}

\begin{definition}\label{def:Holder}
A proper open subset $U$ of a~metric space $(\mathbb{X},\rho)$
is a~\textit{H\"older domain} if there exists a~constant
$H>0$ such that for every $x\in U$
we can find a~path $\gamma$
joining $x$ with the distinguished point $x_*\in U$
satisfying
$$
\int_{\gamma} \frac{ds}{\rho_U(\gamma(s))}
\leqslant H\ln\biggl(\frac{H}{\rho_U(x)}\biggr),
$$
where $ds$ is the arc-length measure.
\end{definition}

The reader may recognize the above integral as the quasihyperbolic length of
$\gamma$.
H\"older domains are also known as domains satisfying a~quasihyperbolic
boundary condition.

It is easy to verify that every John domain is  a~H\"older domain.
Note that the notions of John and H\"older domains
are independent of the choice of the equivalent metrics.

Theorem 1 is a~particular case of the following

\begin{theorem}\label{th:stab_Hol}
Consider a~H\"older domain~$U$  on the Heisenberg group
$\mathbb{H}^n$, $n>1$. For every
$f\in I(1+\varepsilon, U)$
there exists an~isometry $\theta$ satisfying
$$
\int_{U}\exp\Bigl(\frac{N_1|D_hf(x)-D_h\theta(x)|}{\varepsilon}\Bigr)\,dx
\leqslant 16|U|.
$$
The constant $N_1$ depends on $n$, $H$, and $\rho_U(x_*)/\diam(U)$.
\end{theorem}

\begin{lemma}\label{lem:chain}
Suppose $U$ is a~H\"older or John domain in a~metric space $(\mathbb{X},\rho)$.
Then, for every point $x\in U$
there is a~chain of balls $B_i=B(x_i,r_i)$, $i=0,\dots,k$, with
$B_0=B(x_*,\frac{\rho_U(x_*)}{4})$,
satisfying the following conditions:

$(1)$ $x_0=x_*$ and $x_k=x$;

$(2)$ if $0\leqslant i<k$ then
$\frac{7}{9}r_{i+1}\leqslant r_i\leqslant \frac{9}{7} r_{i+1}$
and there is a~ball $G_i\subset B_i\cap B_{i+1}$ with
$r(G_i)=\frac{1}{2}\min\{r_i,r_{i+1}\}$;

$(3)$ $4B_i\subset U$ for all $i=0,\dots,k$;

$(4)$ $k<9H\ln \frac{H}{4r_k}$ if $U$ is a~H\"older domain
and $k<9\frac{\beta}{\alpha}\ln\frac{8\beta}{r_k}$ if $U$ is a~John domain;

$(5)$ $B_k\subset (1+5\frac{\beta}{\alpha})B_i$ and
$B_k\subset (3+10\frac{\beta}{\alpha})G_i$
for all $i=0,\dots,k-1$
if $U$ is a~John domain.
\end{lemma}

\begin{proof}
Fix a~point $x\in U$.
Construct a~chain  $B_0,B_1,\dots, B_k$ of
balls
$B_i=B(x_i,r_i)$ with $r_i=\frac{\dist(x_i,\partial U)}{4}$
for all $i=0,\ldots, k$
and $x_0=x_*$, $x_k=x$.

Thus, we must find the number $k$ and the points $x_1\dots,x_{k-1}$.
Consider a~rectifiable curve~$\gamma$ joining $x$
with~$x_*$ and satisfying the conditions of Definition \ref{def:John}
or Definition \ref{def:Holder}. Parameterize $\gamma$ by arc length.
Put $s_0=l$ and $x_0=x_*=\gamma(l)=\gamma(s_0)$.
Assume by induction that $x_0,\dots,x_i$ are known and
 put $x_{i+1}=\gamma(s_{i+1})$,
where $s_{i+1}=\inf\{s : \gamma\vert_{(s,s_i]}\subset B(x_i,\frac{r_i}{2})\}$.
The process stops on step $j$ if the ball
$\frac{1}{2}B_j$ intersects the ball
$B(x,\frac{\dist(x,\partial U)}{8})$; then, put $j=k-1$.

Conditions (1) and (3) obviously hold.

(2) By construction, $4r_i=\dist(x_i,\partial U)$ for $i=1,\dots,k$,
and $\rho(x_i,x_{i+1})\leqslant \frac{r_i}{2}+\frac{r_{i+1}}{2}$ for
$i=0,\dots, k-1$. Hence,
$4r_i\leqslant
4r_{i+1}+\frac{r_i}{2}+\frac{r_{i+1}}{2}
$
and
$
4r_{i+1}\leqslant
4r_{i}+\frac{r_i}{2}+\frac{r_{i+1}}{2}.
$
Therefore,
$\frac{7}{9}r_{i+1}\leqslant r_i\leqslant \frac{9}{7}r_{i+1}$.
Since
$\frac{1}{2}B_i \cap \frac{1}{2} B_{i+1}\neq \varnothing$
for $i=0,\dots, k-1$,
there is a~ball $G_i$
such that $G_i\subset B_i\cap B_{i+1}$, $x(G_i)\in\frac{1}{2}B_i \cap \frac{1}{2} B_{i+1}$,
and $r(G_i)=\frac{1}{2}\min\{r_i,r_{i+1}\}$.

(4) By construction, $\frac{r_i}{2}=\rho(\gamma(s_i),\gamma(s_{i+1}))
\leqslant l(\gamma\vert_{[s_{i+1},s_i})$
for all $i=0,\ldots,k-2$
and $\frac{r_{k}}{2}<\rho(\gamma(s_{k-1}),\gamma(0))\leqslant l(\gamma\vert_{[0,s_{k-1}]})$.
Hence,
$\sum_{i=0}^{k-2}r_i+r_k\leqslant 2 l$.
If $y\in \frac{1}{2}\overline{B}_i$ then
$\rho_U(y)\leqslant \rho_U(x_i)+\rho(x_i,y)\leqslant
\frac{9}{2}r_i$
for all $i=0,\dots,k$.
Thus,
$$
\int_{\gamma([s_{i+1},s_i])}\frac{ds}{\rho_U(\gamma(s))}
\geqslant
\int_{\gamma([s_{i+1},s_i])} \frac{2ds}{9r_i}
\geqslant
\frac{1}{9}
\quad
\text{for all } i=0,\dots,k-2
$$
and
$$
\int_{\gamma([0,s_{k-1}])}\frac{ds}{\rho_U(\gamma(s))}
\geqslant
\int_{\gamma([0,s_{k-1}])\cap \frac{1}{2}B_k} \frac{2ds}{9r_k}
\geqslant
\frac{1}{9}.
$$
If $U$ is a~H\"older domain then
$
\frac{k}{9}\leqslant \int_\gamma \frac{ds}{\rho_U(\gamma(s))}
\leqslant H\ln \frac{H}{4r_k}
$.
If $U$ is a~John domain then $\frac{r_k}{2}\leqslant s'\leqslant \frac{l}{\alpha}
\rho_U(\gamma(s'))\leqslant
\frac{9l}{2\alpha}r_k$ for
$s'=\sup\{s : \gamma\vert_{[0,s)}\subset \frac{1}{2}B_k\}<s_{k-1}$.
We have
\begin{multline*}
\frac{k}{9}
\leqslant
\int_\gamma \frac{ds}{\rho_U(\gamma(s))}=
\int_{s'}^l\frac{ds}{\rho_U(\gamma(s))}+
\int_{0}^{s'}\frac{ds}{\rho_U(\gamma(s))}
\leqslant
\int_{s'}^l\frac{l ds}{\alpha s}+
\int_{0}^{s'}\frac{2 ds}{7 r_k}
\\=
\frac{l}{\alpha}(\ln l-\ln s')+\frac{2}{7r_k} s'
\leqslant
\frac{l}{\alpha}\Bigl(\ln l-\ln \Bigl(\frac{r_k}{2}\Bigr)+\frac{9}{7}\Bigr)
\leqslant \frac{\beta}{\alpha}\ln \frac{8\beta}{r_k}.
\end{multline*}

(5) Assume that $U$ is a~John domain.
For $i\in\{0,\dots,k-1\}$
we have
$\rho(x,x_i)\leqslant l(\gamma|_{[0,s_i]})=s_i\leqslant
\frac{\beta}{\alpha}\rho_U(x_i)=4\frac{\beta}{\alpha}r_i$
and
$$4r_k=\rho_U(x)\leqslant \rho(x,x_i)+\rho_U(x_i)
\leqslant \Bigl(\frac{4\beta}{\alpha}+4\Bigr)r_i.$$
This yields
$\rho(x_i,z)\leqslant \rho(x,x_i)+r_k \leqslant \bigl(1+\frac{5\beta}{\alpha}\bigr)r_i$
for every $z\in B_k$.
Thus, $B_k\subset (1+5\frac{\beta}{\alpha}) B_i$.

Suppose $r(G_i)=\frac{r_j}{2}$, where $j$ equals either $i$ or $i+1$.
Then
$$
\rho(x(G_i),z)\leqslant \rho(x(G_i),x_j)+
\rho(x_j,z)\leqslant \Bigl(1+2\Bigl(1+5\frac{\beta}{\alpha}\Bigr)\Bigr)\frac{r_j}{2}
$$
for all $z\in (1+5\frac{\beta}{\alpha}) B_j$. Hence,
$B_k\subset (3+10\frac{\beta}{\alpha})G_i$.
\end{proof}

Below we need the following result asserting that
the boundary of a~H\"older domain or John domain is regular in some sense.

\begin{lemma}\label{lem:OLD}
Given a~H\"older domain  $U$ in $\mathbb{H}^n$,
there is a~constant $0<\tau<1$ depending only on $n$,
$H$, and $\frac{\rho_U(x_*)}{\diam U}$
such that
$\int_U \frac{dx}{\rho_U(x)^\tau}\leqslant \frac{2|U|}{\rho_U(x_*)^\tau}$.

Given a~John domain $U$,
there is a~constant $0<\tau_0<1$ depending only on $n$
such that
$\int_U \frac{dx}{\rho_U(x)^\tau}\leqslant \frac{2|U|}{\alpha^\tau}$
with
$\tau= \tau_0 (\frac{\alpha}{\beta})^{\nu}$.
\end{lemma}

\begin{proof}
The first part (on H\"older domains) is Theorem 3.3 of \cite{buck}.
We estimate~$\tau$ in the case of John domains.

Consider a~countable
family of balls $\mathcal{D}$ covering $U$
such that
$\{\frac{1}{5}D\}_{D\in \mathcal{D}}$ is a~disjoint family,
$\sum_{D\in\mathcal D}\chi_D(x)\leqslant N$ for all $x\in U$, and
$r(D)=\frac{1}{4}\rho_U(x(D))$ for every $D\in \mathcal D$.

Fix a~ball $D\in \mathcal D$. By Lemma \ref{lem:chain},
there is a~chain of balls $B_0,\dots,B_k=D$ that covers the
curve~$\gamma$
joining $x(D)$ and $x_*$.
Since the family $\mathcal D$ covers $U$ there is a~chain of balls $D_0,\dots,D_l=D$
in $\mathcal D$ covering $\gamma$ and satisfying $x_*\in D_0$ and $D_i\cap D_{i+1}\neq\varnothing$
for $i=0,\dots,l-1$. We can take the same $D_0$ for all $D\in \mathcal D$.

Put $D_i=B(y_i,\rho_i)$.
Since $D_i\cap D_{i+1}\neq \varnothing$, it follows that $\frac{3}{5}\rho_{i+1}
\leqslant \rho_i\leqslant \frac{5}{3}\rho_{i+1}$ for $i=0,\dots,l-1$.
Suppose that $\rho_l\leqslant \rho_0$. Then $\rho_0\leqslant (\frac{5}{3})^l \rho_l$
and hence $l\geqslant \log_{\frac{5}{3}}\frac{\rho_0}{\rho_l}$.

The chains $\{B_j\}$ and $\{D_i\}$ cover $\gamma$.
Hence, each $D_i$ intersects some ball $B_j$.
It follows that
$\frac{3}{5}r_j\leqslant \rho_i\leqslant \frac{5}{3} r_j$.
Therefore, for $y\in D$ we have
$$
\rho(y,y_i)\leqslant \rho(y,x_j)+\rho(x_j,y_i)\leqslant (1+5\frac{\beta}{\alpha})r_j+r_j+\rho_i
\leqslant \frac{\rho_i}{3}(13+25\frac{\beta}{\alpha})<13 \frac{\beta}{\alpha}\rho_i.
$$
Thus, $D\subset h D_i$ for all $i=0,\dots,l$
with $h=13\frac{\beta}{\alpha}$.
Putting $S(x)=\sum_{D\in \mathcal D} \chi_{hD}(x)$,
we obtain $S(x)\geqslant\log_{\frac{5}{3}}\frac{\rho_0}{r(D)}$ if $x\in D$
and $r(D)\leqslant \rho_0$.

It is known (see  \cite[Lemma 3.4]{buck} for example),
that there exists a~constant $C>1$ depending only on $n$
such that
$\|S\|_{p,\mathbb{H}^n}\leqslant C p h^\nu
\|\sum_{D\in\mathcal D}\chi_{\frac{1}{5}D}\|_{p,\mathbb{H}^n}
\leqslant C p h^\nu |U|^{1/p}$
for all $p\geqslant 1$.
Hence
$$
\bigl\|e^{aS}\bigr\|_{1,U}\leqslant
|U|+\sum_{m>1}\frac{a^m\|S\|^m_{m,U}}{m!}
\leqslant
|U|
\Bigl(
1+\sum_{m>1}\frac{(Ch^\nu m a)^m}{m!}\Bigr)
\leqslant 2|U|
\quad
\text{if }a=\frac{1}{2C h^\nu e}.
$$

For $b=a(\ln\frac{5}{3})^{-1}>0$, we have
\begin{multline*}
\int_U\frac{dx}{\rho_U(x)^b}
\leqslant
\sum_{D\in \mathcal D} \frac{|D|}{3^b r(D)^b}
\leqslant
\sum_{D\in \mathcal D, r(D)\leqslant \rho_0}
\frac{5^\nu}{3^b \rho_0^b}\int_{\frac{1}{5}D}
e^{b S(x)\ln\frac{5}{3}}\,dx+
\sum_{D\in \mathcal D, r(D)>\rho_0}
\frac{|D|}{3^b \rho_0^b}
\\
\leqslant
\frac{5^\nu}{3^b \rho_0^b}\int_{U}
e^{b S(x)\ln\frac{5}{3}}\,dx
+
\frac{5^\nu}{3^b \rho_0^b}|U|
\leqslant
\frac{3\cdot 5^\nu 5^b}{3^b \alpha^b}
|U|
\leqslant
5^{\nu+1} |U|\frac{1}{\alpha^b}.
\end{multline*}
Here we have used the fact that $\alpha\leqslant \rho_U(x_*)\leqslant
\rho(y_0,x_*)+\rho_U(y_0)\leqslant 5 \rho_0$
and the inequality $b\leqslant 1$.
Applying H\"older's inequality, we obtain the desired inequality
for $\tau=\frac{b}{3\nu}=\tau_0(\frac{\alpha}{\beta})^\nu$
where
$\tau_0=\frac{1}{6 C \nu 13^\nu e\ln\frac{5}{3}}$.
\end{proof}

\begin{proof}[Proof of Theorem {\rm\ref{th:stab_Hol}}]
Consider a~mapping $f\in I(1+\varepsilon, U)$, where $U$ is a~H\"older domain.
Put $F(x)=D_h f(x)$.
By Lemma \ref{the:qualitative_local}
for every ball $B$ with $\frac{10}{3}B\subset U$
there is a~unitary matrix $A_B$ such that
$$
\int_B |F(x)-A_B|\,dx\leqslant
|B|^{1/2}
\biggl(\int_B|F(x)-A_B|^2\,dx\biggr)^{1/2}
\leqslant \sigma |B|
$$
with $\sigma = c_2\varepsilon$.
Thus, it is easy to see that $D_h f$ is a~$BMO$ mapping.
By \cite[Theorem 2.2]{buck},
$$\int_{B'}\exp\bigl(C_1\sigma^{-1}|F(x)-F_{B'}|\bigr)dx\leqslant 16|B'|,
$$
where $B'=\frac{1}{2}B$, $C_1=\frac{1}{12}$, and $F_{B'}$ is the mean value of $F$ over the ball $B'$. The proof of this fact goes
along  the same lines as the proof of the classical John--Nirenberg Theorem.
Consequently,
$$
\int_{B'}
e^{C_1 \sigma^{-1}|F(x)-A_B|}dx
\leqslant
e^{C_1 \sigma^{-1}|F_{B'}-A_B|}
\int_{B'}
e^{C_1 \sigma^{-1}|F(x)-F_{B'}|}dx
\leqslant
16e^{2^{\nu }C_1}|B'|.
$$

Consider the family of balls $\{B(x,\frac{\dist(x,\partial U)}{8})\}_{x\in U}$.
We can choose a~countable subfamily
$\mathcal{F}$
such that
$\bigcup_{B\in \mathcal{F}} B=U$
and
$\{\frac{1}{5}B\mid B\in \mathcal{F}\}$ is a~disjoint family.

Put $A_*=A_{B_0}$, where $B_0=B(x_*,\frac{\dist(x_*,\partial U)}{4})$.
For every $B\in \mathcal{F}$,
there is a~chain of balls $B_0,\dots,B_k=2B$ satisfying conditions~(1)--(4)
of Lemma~\ref{lem:chain}.
Obviously,
$$
|A_{B_i}-A_{B_{i+1}}|\leqslant
|G_i|^{-1} \int_{B_i} |F(x)-A_{B_{i}}|\,dx+
|G_i|^{-1}\int_{B_{i+1}} |F(x)-A_{B_{i+1}}|\,dx
\leqslant
\sigma C_2
$$
with $C_2=2^\nu\bigl(1+\bigl(\frac{9}{7}\bigr)^\nu\bigr)$.
For $0<C_3<C_1$, it follows that
$$
\int_B e^{C_3\sigma^{-1}|F(x)-A_*|} dx
\leqslant
\prod_{i=0}^{k-1} e^{C_3\sigma^{-1}|A_{B_i}-A_{B_{i+1}}|}
\int_B e^{C_3\sigma^{-1}|F(x)-A_{B_k}|}
dx
\leqslant
16e^{kC_2 C_3}e^{2^{\nu }C_1}|B|
$$
Applying $k\leqslant 9H\ln\frac{H}{8r(B)}$, we obtain
\begin{equation}
\label{eq:proof_th_2}
\int_U
e^{C_3\sigma^{-1}|F(x)-A_*|}
dx
\leqslant
\sum_{B\in\mathcal{F}}
\int_B
e^{C_3\sigma^{-1}|F(x)-A_*|}
dx
\leqslant
C_4
\int_U \frac{dx}{\rho_U(x)^{9H C_2C_3}}
<\frac{2C_4}{\rho_U(x_*)^\tau}|U|
\end{equation}
if $C_3$ is small enough so that $9HC_2C_3\leqslant \tau$, where
$\tau$ is as in Lemma \ref{lem:OLD}.
\end{proof}

\begin{proof}[Proof of Theorem {\rm\ref{th:stab_John}}]
Consider a~John domain $U$ with a distinguished point $x_*$
and a~mapping $f$ of class $I(1+\varepsilon,U)$. Put
$B_*=B(x_*,r_*)$, where $r_*=\frac{\dist(x_*,\partial U)}{4}$.

The proof of the first assertion
follows verbatim the proof of Theorem~\ref{th:stab_Hol}
till relation \eqref{eq:proof_th_2}.
We rearrange \eqref{eq:proof_th_2} as
$$
\int_U
e^{C_3\sigma^{-1}|D_hf(x)-A_*|}
dx
\leqslant
C_4 \beta^{9\frac{\beta}{\alpha} C_2 C_3}
\int_U \frac{dx}{\rho_U(x)^{9\frac{\beta}{\alpha} C_2 C_3}}
\leqslant \frac{2C_4 \beta^\tau |U|}{\alpha^\tau}
\leqslant 4C_4 |U|
$$
if $9\frac{\beta}{\alpha} C_2 C_3<\tau=\tau_0(\frac{\alpha}{\beta})^\nu$
where $C_4=64\cdot 2^{2^\nu C_1}5^\nu$.
Here we use the fact that $(\frac{\beta}{\alpha})^{\frac{\alpha}{\beta}}<2$
and, consequently, $(\frac{\beta}{\alpha})^\tau<2$
since $\tau_0<1$.
By H\"older's inequality, we obtain the desired inequality with
$N_1=\frac{C_3}{2+\nu+2^{\nu-3}}=C'(\frac{\alpha}{\beta})^{\nu+1}$.

Let us now prove the second assertion.

Consider a~point $x\in U$ and the chain $B_0,\dots, B_k$ of Lemma \ref{lem:chain}.
Since balls are John domains and, consequently, H\"older domains,
Theorem \ref{th:stab_Hol} implies that, for each $i=0,\dots,k$,
there is an~isometry $\theta_i$
such that
$$
\int_{4B_i}\exp\biggl\{\frac{N_1|D_h f(x)-D_h\theta_i(x)|}{\varepsilon}\biggr\}
dx
\leqslant 16|4B_i|.
$$
Hence,
$
\|D_h f-D_h\theta_i\|_{\nu+1,4B_i}
\leqslant
\bigl(\frac{\varepsilon}{N_1}\bigr)
(16 |4B_i|)^{1/(\nu+1)}
$,
and, by Lemma \ref{lem:8}, we conclude that
$\rho(f(y),\theta_i(y))\leqslant \omega r_i$
for all $y\in B_i$
with $\omega=C_1 (\sqrt{\varepsilon}+\varepsilon)$.

We have
$$
\rho(f(x),\theta_0(x))\leqslant
\rho(f(x),\theta_k(x))+
\sum_{i=0}^{k-1} \rho(\theta_i(x),\theta_{i+1}(x))
$$
and
$$
\rho(\theta_i(y),\theta_{i+1}(y))\leqslant
\rho(f(y),\theta_i(y))
+
\rho(f(y),\theta_{i+1}(y))
\leqslant
\omega(r_i+r_{i+1})
\leqslant
\frac{32}{7} \omega r(G_i)
$$
for every $y\in G_i$.
Consider the case $\frac{32}{7} \omega <\frac{1}{2}$.
Lemma~\ref{lem:isom} yields
$$
\rho(\theta_i(y),\theta_{i-1}(y))\leqslant
C_2\omega r(G_i),
$$
where $C_2=\frac{160}{7} (3+10\frac{\beta}{\alpha})$,
for all $y\in B(x,r)=B_k\subset (3+10\frac{\beta}{\alpha})G_i$ and all $i=1,\ldots,k-1$,

Since $r(G_i)\leqslant \frac{1}{2}r_i$ for $i=0,\dots,k-2$ and
$r(G_{k-1})\leqslant \frac{1}{2}r_k$,
it follows that
$$
\rho(f(x),\theta_0(x))\leqslant
\omega r_k
+C_2 \omega \sum_{i=0}^{k-1}r(G_i)
\leqslant
C_3 \omega \beta,
$$
where $C_3=\frac{1}{4}+C_2$.

Consider the case $\frac{32}{7} \omega \geqslant\frac{1}{2}$.
Without loss of generality we may assume  that
$\varepsilon\geqslant \varepsilon_4$ for a~constant $\varepsilon_4$ of the proof of Lemma \ref{the:qualitative_local},
and
$\theta_i$ is just the left translation satisfying
$\theta_i(x_i)=f(x_i)$.
Thus we can apply Lemma \ref{lem:9}. The theorem follows in the same way.
\end{proof}

\section{Appendix}

\subsection{Application of the embedding theorem}
\begin{lemma}\label{lem:8}
Let $f\in I(1+\varepsilon,B(a,r))$
and
$$
\|D_h f-I\|_{p,B(a,r)}
\leqslant \varepsilon|B(a,r)|^{1/p}.
$$
If $p>\nu$ and $f(a)=a$ then
$$
\rho(f(x),x)\leqslant
C r (\sqrt{\varepsilon}+\varepsilon)
\quad
\text{for all }x\in  B(a,s r)
$$
with $s\in (0,1)$.
The constant $C$ depends only on $n$, $p$, and $s$.
\end{lemma}

\begin{proof}
Put $B=B(a,r)$. Denote the first $2n$
coordinates of $x^{-1}\cdot f(x)$ by  $\psi(x)$
and the last coordinate of $x^{-1}\cdot f(x)$ by $\chi(x)$.
Estimate $\nabla_{\mathcal{L}}\psi_i(x)$ for all
$i=1,\dots,2n$ and $\nabla_{\mathcal{L}}\chi(x)$.
Clearly,
$$
\|\nabla_\mathcal{L}\psi_i\|_{p,B}=
\|\nabla_\mathcal{L} f _i-\nabla_\mathcal{L}x_i\|_{p,B}
\le
\|D_h f-I\|_{p,B},
\quad
i=1,\dots,2n.
$$
The embedding theorem (see \cite{gar-nhieu} for example) yields
$$
|\psi_k(x)|
\le C_1 r^{1-\nu/p}\|\nabla_{\mathcal{L}}\psi_k\|_{p,B}
\le
C_2\varepsilon r\quad
\text{for all }x\in \frac{s+1}{2}B,
\quad k=1,\dots,2n.
$$

We have
$$
\chi(x)= f _{2n+1}(x)-x_{2n+1}
+2\sum_{j=1}^n \bigl(x_j  f _{j+n}(x)-
x_{j+n}  f _j(x)\bigr).
$$
The contact condition
$X_i f(x)\in H_{f(x)}\mathbb H^n$ for $i=1,\dots,2n$
yields
$$
X_i f_{2n+1}(x)=2\sum_{j=1}^n f_{j+n}(x)X_i f_j(x)-f_j(x)X_i f_{j+n}(x),
$$
and then we deduce that
$\nabla_\mathcal{L}\chi(x)= 2\bigl((D_h f (x))^t+I\bigr) J \psi(x)$,
where $J$ is the $2n\times 2n$ matrix defined in \eqref{eq:operator_Q}.

Applying the embedding theorem once again, we obtain
$$
|\chi(x)|
\le C_3 r^{1-
\nu/p}\|\nabla_{\mathcal{L}}\chi\|_{p,\frac{s+1}{2}B}
\le
C_4 r
\|\psi\|_{C(\frac{s+1}{2}B)}
(2+\varepsilon)
\le
C_5 r^{2} \varepsilon(2+\varepsilon)
\quad
\text{for all }x\in s B.
$$
Hence,
$\rho(f(x),x)\leqslant
C_6 r(\sqrt{\varepsilon(2+\varepsilon)}+\varepsilon)
\leqslant C_7 r(\sqrt{\varepsilon}+\varepsilon)
$ for all $x\in B(a,s r)$.
\end{proof}

\subsection{Isometries on the balls}
\begin{lemma}\label{lem:isom}
If $\varphi$ is an~isometry on $\mathbb{H}^n$
with $\rho(\varphi(x),x)\leqslant \varepsilon r$ for all $x\in \overline{B(a,r)}\subset \mathbb{H}^n$
with $\varepsilon<1/2$,
then
$
\rho(\varphi(x),x)\leqslant 5 \varepsilon s r
$
for all $x\in \overline{B(a,sr)}$, $s\geqslant 1$.
\end{lemma}

\begin{proof}
Assume that $B(a,r)=B(0,1)$.
Suppose firstly that $\varphi=\iota\circ \pi_\mathbf{a}\circ \varphi_A$ where
$\mathbf{a}=(a,\alpha)\in \mathbb{H}^n$ with
$a\in \mathbb{C}^n$ and $\alpha\in \mathbb{R}$,
as well as $A\in U(n)$.
If $x=0$ then $|a|<1/2$ and $|\alpha|\leqslant 1/4$.
If $z=0$ and $t=1$ then we arrive at a~contradiction:
$$
1/2 \geqslant \rho\bigl(\varphi(0,1),(0,1)\bigr)
=
\rho\bigl((\overline{a},-\alpha-2)\bigr)
\geqslant
\sqrt{2+\alpha}.
$$

Thus, $\varphi=\pi_\mathbf{a}\circ \varphi_A$, where $\mathbf{a}=(a,\alpha)\in \mathbb{H}^n$,
$a\in \mathbb{C}^n$, $\alpha\in \mathbb{R}$,
and $A\in U(n)$.
We have
$$
x^{-1}\cdot \mathbf{a}\cdot \varphi_A x=
(-z+a+Az, \alpha+2\Image\langle a,Az\rangle -2
\Image \langle z, a+Az\rangle).
$$
Clearly, $|\mathbf{a}|=\rho(\varphi(0),0)\leqslant \varepsilon$
and
$
|Az-z|\leqslant |Az-z+a|+|a|\leqslant 2\varepsilon
$.

We have
\begin{multline*}
2|\Image\langle 2a+ A z,z\rangle |\leqslant
|\alpha+2\Image\langle a, A z\rangle -
2\Image\langle z, Az+a\rangle |\\+
|\alpha|+
2|\Image\langle a, A z-z+a\rangle |
\leqslant 4\varepsilon^2
\quad
\text{for all }z\in \mathbb{ C}^n,\ |z|\leqslant 1.
\end{multline*}
Suppose that $ Aa=\xi a~+d$ and $(d,a)=0$ with
$\xi\in \mathbb{ C}$ and $d\in \mathbb{ C}^n$.
Put $z= \gamma a$,
$|z|\leqslant 1$.
Then
$$
|\Image\langle 2a+ A z,z\rangle |=
|\Image\langle 2a+\gamma (\xi a+d),\gamma a\rangle |
=
|a|^2|\Image(2\overline{\gamma}+ \xi|\gamma|^2)|
\leqslant 2\varepsilon^2.
$$
Suppose that $a\neq 0$. For $\gamma=\frac{1}{|a|}$, we infer that
$|\Image\langle 2a+ A z,z\rangle |=|\Image \xi|\leqslant 2\varepsilon^2$.
For $\gamma=\frac{-i}{|a|}$,
$$
2|a|\leqslant
|a|^2|\Image(2\overline{\gamma}+ \xi|\gamma|^2)|
+|\Image \xi|
\leqslant 4\varepsilon^2.$$
Hence,
$$
|\Image\langle A z,z\rangle |
\leqslant
|\Image\langle 2a+ A z,z\rangle|+|\Image\langle 2a,z\rangle|
\leqslant 6\varepsilon^2.
$$
In the case of $a=0$, we obviously have
$
|\Image\langle  A z,z\rangle|
\leqslant 2\varepsilon^2
$.

Consider $y=\delta_s x\in B(0,s)$.
We obtain
$$
y^{-1}\cdot \mathbf{a}\cdot \varphi_A y
=(-sz+a+sAz, \alpha+2\Image\langle a,A s z\rangle -2
\Image \langle sz, a+Asz\rangle).
$$
Then
$
|-sz+a+sAz|\leqslant s|Az-z|+a\leqslant (2s+1)\varepsilon
$
and
\begin{multline*}
|\alpha+2\Image\langle a,A s z\rangle -2
\Image \langle sz, a+Asz\rangle|
\\
\leqslant
|\alpha|+2|\Image\langle a,Asz+sz\rangle|
+2|\Image \langle sz,Asz\rangle|
\leqslant (1+8s+12s^2)\varepsilon^2.
\end{multline*}
Thus,
$
\rho(\pi_\mathbf{a}\circ \varphi_A(y),y)
\leqslant
5s\varepsilon
$.

Now, take an~arbitrary ball $B(a,r)$
and suppose that $\rho(\varphi(x),x)\leqslant \varepsilon r$
on $B(a,r)$.
The isometry $\theta=\delta_{1/r}\circ
\pi_{-a}\circ \varphi \circ \pi_{a}\circ \delta_{r}$ satisfies
 $\rho(\theta(y),y)\leqslant 5s\varepsilon$ for all $y\in B(0,s)$.
Inserting
$x=a\cdot \delta_r y$ for $x\in B(a,sr)$, we obtain the required estimate.
\end{proof}

The following lemma is obvious.
\begin{lemma}\label{lem:9}
If
$\rho(bx,x)\leqslant \varepsilon$ on $B(a,r)$
then
$\rho(bx,x)<3s\varepsilon$
on $B(a,sr)$, $s\geqslant 1$.
\end{lemma}

\bibliographystyle{amsplain}

\end{document}